\documentclass[12pt]{article}
 \usepackage{graphicx,amsmath,amsfonts,amssymb}
\usepackage[numbers]{natbib}
\usepackage[dvips]{color}

\usepackage{cases}

\definecolor{brown}{cmyk}{0,0.83,1,0.6}

\newcommand{\rmn}[1]{\if#11I\else {\if#12I\hspace{-0.12ex}I\hspace{-0.08ex}\else {\if #13I\hspace{-0.12ex}I\hspace{-0.12ex}I\hspace{-1.5ex} \else{\if#14I\hspace{-0.16ex}V\hspace{-2.4ex} \else{\if#15V\hspace{-3.0ex} \else{\if#16V\hspace{-0.12ex}I\hspace{-3.0ex} \else{\if#16V\hspace{-0.12ex}I  \else{\if#17V\hspace{-0.12ex}I\hspace{-0.12ex}I\hspace{-3.0ex}  \else{\if#18V\hspace{-0.12ex}I\hspace{-0.12ex}I\hspace{-0.12ex}I\hspace{-3.8ex} \else{\if#19I\hspace{-0.12ex}X\hspace{-3.8ex} \fi}\fi} \fi}\fi}\fi}\fi} \fi} \fi} \fi}\fi}

 \newcommand{\vc}[1]{{\boldsymbol #1}}
 \newcommand{\vcn}[1]{{\bf #1}}
 
 \newcommand{\sr}[1]{{\mathcal #1}}
 \newcommand{\dd}[1]{\mathbb{#1}}

 \newcommand{\eqn}[1]{(\ref{eqn:#1})}
 \newcommand{\lem}[1]{Lemma~\ref{lem:#1}}
 \newcommand{\cor}[1]{Corollary~\ref{cor:#1}}
 \newcommand{\thr}[1]{Theorem~\ref{thr:#1}}

 \newcommand{\rem}[1]{Remark~\ref{rem:#1}}
 \newcommand{\fig}[1]{Figure~\ref{fig:#1}}
\newcommand{\app}[1]{Appendix~\ref{app:#1}}
\newcommand{\sectn}[1]{Section~\ref{sec:#1}}

\newcommand{\lemt}[1]{\ref{lem:#1}}

 \newcommand{\pend}{\hfill \thicklines \framebox(6.6,6.6)[l]{}}
 \newenvironment{proof}{\noindent {\sc  Proof.} \rm}{\pend}
 \newenvironment{proof*}[1]{\noindent {\sc  #1} \rm}{\pend}

 \newtheorem{theorem}{Theorem}[section]
 \newtheorem{lemma}{Lemma}[section]

 \newtheorem{remark}{Remark}[section]

 \newtheorem{corollary}{Corollary}[section]
 \newtheorem{definition}{Definition}[section]

 \newcommand{\setnewcounter} {
 \setcounter{subsection}{0}
 \setcounter{equation}{0}
 \setcounter{conjecture}{0}
 \setcounter{assumption}{0}
 \setcounter{question}{0}
 \setcounter{definition}{0}
 \setcounter{theorem}{0}
 \setcounter{corollary}{0}
 \setcounter{lemma}{0}
 \setcounter{proposition}{0}
 \setcounter{remark}{0}
}

\newenvironment{mylist}[1]{\begin{list}{}
{\setlength{\itemindent}{#1mm}}
{\setlength{\itemsep}{0ex plus 0.2ex}}
{\setlength{\parsep}{0.5ex plus 0.2ex}}
{\setlength{\labelwidth}{10mm}}
}{\end{list}}

 \title{\bf Structure-reversibility of a two dimensional reflecting random walk and its application to queueing network}
\author{Masahiro Kobayashi$^{a}$, Masakiyo Miyazawa$^{a}$ and Hiroshi Shimizu$^{b}$
\\ {\small $^{a}$Department of Information Sciences, Tokyo University of Science}
\\  
{\small $^{b}$Nihon Unisys, Ltd.}
}
\date{Revised version, Mar. 13, 2014}

\begin{document}
\maketitle
\begin{abstract}
We consider a two dimensional reflecting random walk on the nonnegative integer quadrant. It is assumed that this reflecting random walk has skip free transitions. We are concerned with its time reversed process assuming that the stationary distribution exists. In general, the time reversed process may not be a reflecting random walk. In this paper, we derive necessary and sufficient conditions for the time reversed process also to be a reflecting random walk. These conditions are different from but closely related to the product form of the stationary distribution. 
\end{abstract}

\section{Introduction}
\label{sec:Intro}
We consider a two dimensional reflecting random walk on the nonnegative integer quadrant. We are interested in its stationary distribution for queueing applications. This stationary distribution is generally hard to analytically get, and recent research interests have been directed to its tail asymptotics. We now have good pictures for those tail asymptotics (e.g., see \cite{Miya2011}), but many other characteristics like moments are not available. In this paper, we look at this problem up side down. Namely, we aim to find a class of the reflecting random walks whose stationary distributions are obtained in closed form.

For this, we use a time reversed process for the reflecting random walk, and expect that the stationary distribution is analytically obtained when the time reversed process is also a reflecting random walk. \citet{Kell1979} pioneered to use this time reversed idea for deriving so called product form solutions for various queueing network models. Here, the stationary distribution is said to be a product form solution if it is the product of marginal stationary distributions. This product form solution has been further studied (see, e.g., \cite{ChaoMiyaPine1999,Kell1979,Serf1999} and references therein). However, they are limited in use for applications.

While those traditional approaches use the Markov chains which are specialized to queueing models, we here use a different class of Markov chains. Namely, we take two dimensional reflecting random walks for this class motivated by \cite{Miya2012}. They may be interpreted as queueing models, but our intension is to consider the time reversed processes within a class of two dimensional reflecting random walks. In this way, we study a class of the reflecting random walks which have tractable stationary distributions.

To make clear our arguments, we introduce the notion of structure-reversibility. A reflecting random walk is said to be structure-reversible if it has a stationary distribution and if its time reversed process under this stationary distribution is a reflecting random walk, where its transition probabilities may not be identical with those of the original reflecting random walk. We then derive necessary and sufficient conditions for the reflecting random walk to be structure-reversible under appropriate assumptions. In this derivation, the stationary distribution is simultaneously obtained. This stationary distribution has product form in the interior of the quadrant but may not be of product form on the boundary.

It is notable that our stationary distribution is closely related to the one which was recently obtained by \citet{LatoMiya2013}. They derived it by characterizing a class of the reflecting random walks whose stationary distribution has  product form. In the interior of the quadrant, this two dimensional distribution has geometric marginals whose rates are identical with those of ours. We have geometric interpretations to characterize such decay rates similarly to those of \cite{LatoMiya2013}. However, the stationary distribution under structure-reversibility is not necessary to have product form as we already remarked (see \cor{reversibility_product form} and \rem{LatoMiya2013} for further discussions). We also note that the terminology, ``structure-reversibility'', is used for queueing networks with batch movements in \citet{Miya1997}. Its spirit is the same as the one of the present paper, but the classes of models to be applied are different.

This paper is made up by five sections. We formally define our reflecting random walk and its time reversed process in \sectn{ran_rev}. In \sectn{necessary}, we derive necessary and sufficient conditions of structure-reversibility assuming the existence of the stationary distribution.
It is not easy to check those structure-reversibility conditions since they use the stationary distribution. In \sectn{Geometric_conditions},  we obtain the structure-reversibility conditions not using the stationary distribution, that is, only using modeling parameters.  
We finally in Section 5 discuss a class of queueing networks flexible at boundaries, and give examples to be structure-reversible. 

\section{Two dimensional reflecting random walk and its time-reversed process}
\label{sec:ran_rev}
\setnewcounter
In this section, we briefly introduce a two dimensional reflecting random walk and its time-reversed process. We use the following notations.
\begin{eqnarray*}
\begin{array}{llll}
\mathbb{Z} = \mbox{the set of all integers},
\quad \mathbb{Z}_{+} = \{ i \in \mathbb{Z}; i \ge 0\},
\quad \mathbb{R} = \mbox{the set of all real numbers}.
\end{array}
\end{eqnarray*} 
Let $\sr{S} \equiv \mathbb{Z}_{+}^{2}$ be a state space for the reflecting random walk. We partition it into following subsets.
\begin{eqnarray*}
&&\sr{S}_{0} = \{(0,0)\}, \quad \sr{S}_{1} = \{(i,0) \in \sr{S}; i \ge 1\},\\
&&\sr{S}_{2} = \{(0,j) \in \sr{S}; j \ge 1\}, \quad \sr{S}_{+} = \{(i,j) \in \sr{S}; i,j \ge 1\}.
\end{eqnarray*}
Let $\partial \sr{S} = \cup_{i=0}^{2} \sr{S}_{i}$. Clearly, $\sr{S} = \sr{S}_{+} \cup \partial \sr{S}$. The subsets $\sr{S}_{+}$ and $\partial \sr{S}$ are called an interior and a boundary, and $\sr{S}_{i}$ is called a boundary face for $i = 0,1,2$.

Let $\vc{X}^{(+)} \equiv (X_{1}^{(+)},X_{2}^{(+)})$ be a random vector taking value in $\dd{U} \equiv \{-1,0,1\}^{2}$, 
and for $i=0,1,2$, $\vc{X}^{(i)} \equiv (X^{(i)}_{1}, X^{(i)}_{2})$ be a random vector taking values in $\dd{U}$ such that $\vc{X}^{(0)} \ge (0,0)$, $X^{(1)}_{2} \ge 0$ and $X^{(2)}_{1} \ge 0$. Define $\{\vc{Z}_{\ell}; \ell \in \dd{Z}_{+}\}$ as the Markov chain with state space $\sr{S}$ and the following transition probabilities.
\begin{eqnarray}
\label{eqn:reflecting_random_walk}
\mathbb{P}(\vc{Z}_{\ell + 1} = \vc{n}' | \vc{Z}_{\ell} = \vc{n}) = \left\{
\begin{array}{ll}
\mathbb{P}(\vc{X}^{(+)} = \vc{n}' - \vc{n}), & \vc{n} \in \sr{S}_{+}, \vc{n}' \in \sr{S},\\
\mathbb{P}(\vc{X}^{(i)} = \vc{n}' - \vc{n}), &i=0,1,2, \vc{n} \in \sr{S}_{i}, \vc{n}' \in \sr{S}.
\end{array}
\right.
\end{eqnarray}
By this definition, $\{\vc{Z}_{\ell}\}$ is skip free, and has homogeneous transition probabilities in each boundary face $\sr{S}_{i}$ and the interior $\sr{S}_{+}$. We refer to this Markov chain $\{\vc{Z}_{\ell}\}$ as a two dimensional reflecting random walk. Throughout this paper, we tentatively assume that 
\begin{mylist}{0}
\item[(i)] $\{\vc{Z}_{\ell}\}$ is irreducible.
\end{mylist}

The modeling primitives of this reflecting random walk are given by the four sets of the following probability distributions:
\begin{eqnarray*}
&& p_{ij}^{(0)} = \mathbb{P}(\vc{X}^{(0)} = (i,j)), \qquad i,j=0,1,\\
&& p_{ij}^{(1)} = \mathbb{P}(\vc{X}^{(1)} = (i,j)), \qquad i=0,\pm1, j=0,1,\\
&& p_{ij}^{(2)} = \mathbb{P}(\vc{X}^{(2)} = (i,j)), \qquad i=0,1, j=0,\pm1,\\
&& p_{ij}^{(+)} = \mathbb{P}(\vc{X}^{(+)} = (i,j)), \qquad i, j =0,\pm1.
\end{eqnarray*}
The transition diagram of $\{\vc{Z}_{\ell}\}$ is illustrated in \fig{TransitionProbability}.
\begin{figure}[h]
\begin{center}
\includegraphics[height=0.25\textheight]{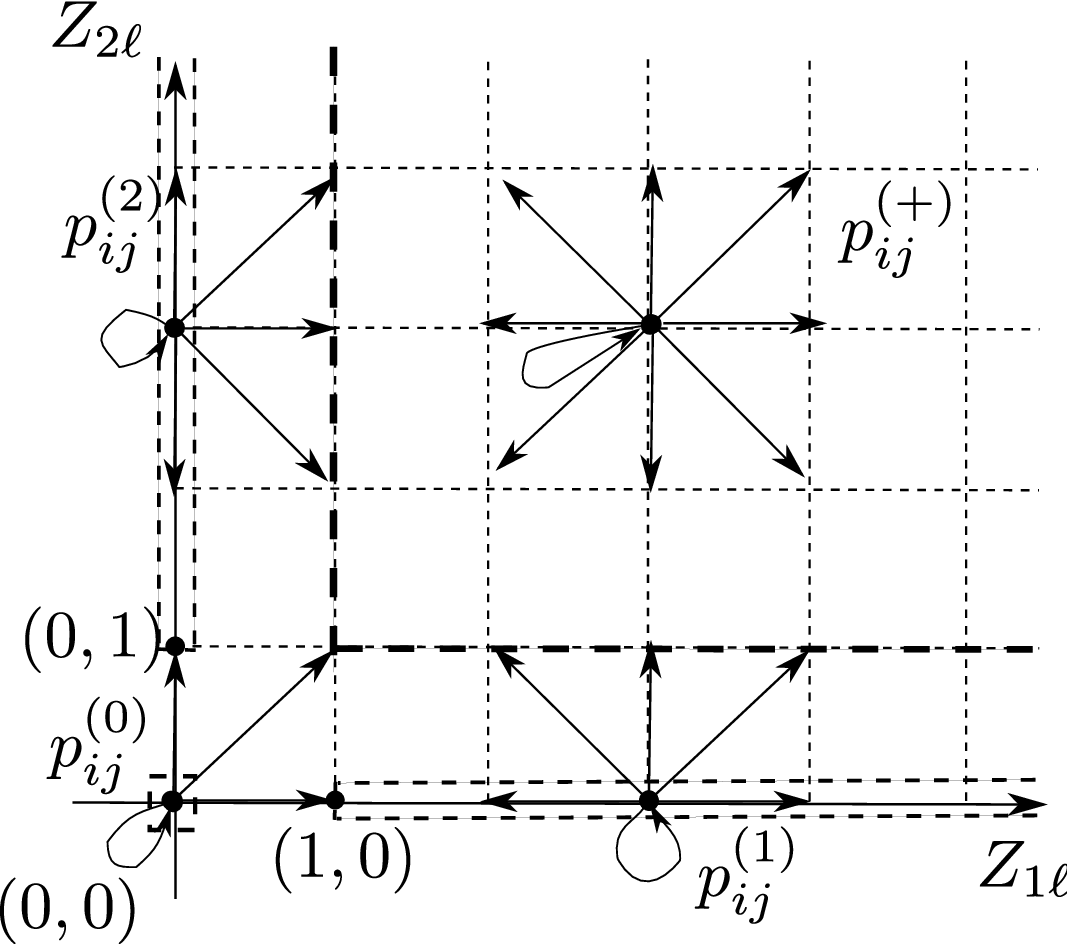}
\end{center}
\caption{Transition diagram of $\{\vc{Z}_{\ell}\}$}
\label{fig:TransitionProbability}
\end{figure}

For the modeling parameter $p_{ij}^{(+)}$, we add the irreducibility assumption.
Let $\{\vc{Y}_{\ell} \in \mathbb{Z}^{2}; \ell \in \mathbb{Z}_{+}\}$ be a two dimensional random walk removing the boundary $\partial {\sr S}$ of $\{\vc{Z}_{\ell}\}$. We note that the distribution  of increments for $\{\vc{Y}_{\ell}\}$ is identical with $p_{ij}^{(+)}$.
For this random walk, we assume the following condition.
\begin{mylist}{0}
\item[(ii)] $\{\vc{Y}_{\ell};\ell \in \mathbb{Z}_{+}\}$ is irreducible.
\end{mylist}
Of course, if $\{\vc{Y}_{\ell}\}$ is not irreducible, the irreducible condition (i) may be still satisfied because of the reflection at the boundary. 
We first note the following fact.

\begin{lemma}
\label{lem:ran_irr}
{\rm
Under the condition (ii), at least one of  $p_{10}^{(+)}$, $p_{(-1)0}^{(+)}$, $p_{01}^{(+)}$ and $p_{0(-1)}^{(+)}$ is positive.
}
\end{lemma}

\begin{proof}
Suppose that our claim is not true, that is, $p_{10}^{(+)}=p_{(-1)0}^{(+)}=p_{01}^{(+)}=p_{0(-1)}^{(+)}=0$. But then, for any $n_{1},n_{2} \in \mathbb{Z}$, if $\vc{Y}_{0} = (n_{1},n_{2})$, then $\vc{Y}_{\ell}$ does not arrive at $(n_{1}+1,n_{2})$ and $(n_{1},n_{2}+1)$ for any $\ell \in \mathbb{Z}_{+}$. This is a contradiction for irreducibility of $\{\vc{Y}_{\ell}\}$. 
\end{proof}

Our problem is to find necessary and sufficient conditions for $\{{\vc{Z}}_{\ell}\}$ under which its time reversal is a reflecting random walk on $\sr{S}$.
For this, we  assume that 
\begin{itemize}
\item[(iii)] $\{\vc{Z}_{\ell}\}$ has the stationary distribution.
\end{itemize}
We denote it by $\pi$. Then, we can construct the stationary Markov chain $\{\vc{Z}_{\ell}; \ell \in \mathbb{Z}\}$ starting from $-\infty$ with $\vc{Z}_{0}$ subject to $\pi$. We define the time-reversed process  $\{\tilde{\vc{Z}}_{\ell}; \ell \in \mathbb{Z}\}$ by
\begin{eqnarray*}
\tilde{\vc{Z}}_{\ell} = \vc{Z}_{-\ell}, \quad \forall \ell \in \mathbb{Z}.
\end{eqnarray*}
It is easy to see that $\{\tilde{\vc{Z}}_{\ell}\}$ is the Markov chain, and transition probability of $\{\tilde{\vc{Z}}_{\ell}\}$ is given by 
\begin{eqnarray}
\label{eqn:transition_of_reversed}
\mathbb{P}(\tilde{\vc{Z}}_{\ell + 1} = \vc{n}' | \tilde{\vc{Z}}_{\ell} = \vc{n}) = \frac{\pi(\vc{n}')}{\pi(\vc{n})} \mathbb{P}(\vc{Z}_{\ell+1} = \vc{n} | \vc{Z}_{\ell} = \vc{n}'), \quad \vc{n},\vc{n}' \in \sr{S}.
\end{eqnarray}
The Markov chain $\{\tilde{\vc{Z}}_{\ell}\}$ is referred to as a time reversed process under $\pi$ (see, e.g., \cite{Asmu2003}). 
\begin{remark}
{\rm 
In \eqn{transition_of_reversed}, we may not require that the $\pi$ is the stationary distribution of $\{\vc{Z}_{\ell}\}$. It is the stationary distribution if and only if 
\begin{eqnarray*}
\label{eqn:necessary_sufficient_stationary}
\sum_{\vc{n}' \in {\sr S}} \mathbb{P}(\tilde{\vc{Z}}_{\ell + 1} = \vc{n}' | \tilde{\vc{Z}}_{\ell} = \vc{n}) = 1,
\end{eqnarray*}
for all $\vc{n} \in {\sr S}$ (see, e.g., \cite{Kell1979,Miya2012}).
}
\end{remark}

\section{Characterization of structure-reversibility}
\setnewcounter
\label{sec:necessary}

A Markov chain is said to be reversible if its time reversed process is stochastically identical with the original Markov chain.
However, this condition is too strong (see, e.g., \cite{Miya2012} for queueing models).
Instead of this reversibility, we consider weaker concept of reversibility for a reflecting random walk.
\begin{definition}
\label{def:quasi-reversibility}
{\rm
The reflecting random walk  $\{\vc{Z}_{\ell}\}$ is said to be structure-reversible if it has a stationary distribution and if its time-reversed process  $\{\tilde{\vc{Z}}_{\ell}\}$ is also a reflecting random walk. 
}
\end{definition}
Thus, if $\{\vc{Z}_{\ell}\}$ has structure-reversibility, then the transition probabilities of $\{\tilde{\vc{Z}}_{\ell}\}$ are given by
\begin{eqnarray}
\label{eqn:RRRW_equation}
\mathbb{P}( \tilde{\vc{Z}}_{\ell + 1} = \vc{n}' | \tilde{\vc{Z}}_{\ell} = \vc{n}) =  \mathbb{P}(\tilde{\vc{X}}^{(i)} = \vc{n}' - \vc{n}), \quad \ell \in \mathbb{Z}, i=0,1,2,+, \vc{n} \in \sr{S}_{i},
\end{eqnarray}
for some random vector $\tilde{\vc{X}}^{(i)} = (\tilde{X}_{1}^{(i)},\tilde{X}_{2}^{(i)})\in \{0,1,-1\}^{2}$.  From the reflection property of $\{\tilde{\vc{Z}}_{\ell}\}$, it is required that $X_{1}^{(i)} \ge 0$ for $i=0,2$ and $X_{2}^{(i)} \ge 0$ for $i=0,1$.
We also note that the distributions of $\vc{X}^{(i)}$ and $\tilde{\vc{X}}^{(i)}$ may not be the same.

We are now ready to present conditions for structure-reversibility.@
\begin{theorem}
\label{thr:reversibility condition}
{\rm
For the reflecting random walk $\{\vc{Z}_{\ell}\}$, assume the conditions (i), (ii) and (iii). Then $\{\vc{Z}_{\ell}\}$ is structure-reversible if and only if the following conditions hold.
\begin{itemize}
\item[(a1)] There exist $c^{(1+)},c^{(2+)} > 0$  such that $c^{(1+)} p_{i1}^{(+)} = p_{i1}^{(1)}$ and  $c^{(2+)} p_{1i}^{(+)} = p_{1i}^{(2)}$ for any $i=0,\pm1$.
\item[(a2)] There exist $c^{(10)}, c^{(20)} \ge 0$ such that $c^{(10)} p_{1j}^{(1)} = p_{1j}^{(0)}$ and $c^{(20)} p_{j1}^{(2)} = p_{j1}^{(0)}$ for any $j=0,1$. If $c^{(10)} = 0$ (resp. $c^{(20)} = 0$), then $p_{1j}^{(1)} = 0$ (resp. $p_{j1}^{(2)} = 0$) for any $j=0,1$. 
\item[(a3)]If both $c^{(10)} > 0$ and $c^{(20)} > 0$,  then $c^{(10)} c^{(1+)} = c^{(20)}c^{(2+)}$.
\item[(a4)] There exist $\eta_{1},\eta_{2} \in (0,1)$ such that
\begin{eqnarray}
\label{eqn:stationary_dist2_1}
&&\pi(n,0) = \eta_{1}^{n-1} \pi(1,0),  \quad n \ge 1,\\
\label{eqn:stationary_dist2_2}
&&\pi(0,n) = \eta_{2}^{n-1} \pi(0,1), \quad n \ge 1,\\
\label{eqn:stationary_dist3}
&&\pi(n_{1},n_{2}) = \eta_{1}^{n_{1}-1} \eta_{2}^{n_{2}-1} \pi(1,1), \quad n_{1},n_{2} \ge 1.
\end{eqnarray}
If $c^{(10)} > 0$, then
\begin{eqnarray}
\label{eqn:stationary_dist1}
&&\pi(1,0) = c^{(10)} \eta_{1} \pi(0,0), \\
\label{eqn:stationary_equation_c20=0}
&&\pi(0,1) = \frac{c^{(10)} c^{(1+)}}{c^{(2+)}} \eta_{2}\pi (0,0),\\
\label{eqn:stationary_dist11}
&&\pi(1,1) = c^{(10)}c^{(1+)} \eta_{1}\eta_{2} \pi(0,0),
\end{eqnarray}
and if $c^{(20)} > 0$, then 
\begin{eqnarray}
\label{eqn:stationary_equation_c10=0}
&&\pi(1, 0) = \frac{c^{(20)} c^{(2+)}}{c^{(1+)}} \eta_{1}\pi (0,0),\\
\label{eqn:stationary_dist1_2}
&&\pi(0,1) = c^{(20)} \eta_{2} \pi(0,0), \\
\label{eqn:stationary_dist11_2}
&&\pi(1,1) = c^{(20)}c^{(2+)} \eta_{1}\eta_{2} \pi(0,0).
\end{eqnarray}
\end{itemize}
}
\end{theorem}

\begin{remark}
\label{rem:irreducibility}
{\rm
 Under the condition (a2), if $c^{(10)} = c^{(20)} = 0$, then $\{\vc{Z}_{\ell}\}$ is not irreducible, that is, the condition (i) is not satisfied. Therefore, at least one of $c^{(10)} > 0$ and  $c^{(20)} > 0$ holds, and we can obtain at least one of \eqn{stationary_dist1}--\eqn{stationary_dist11} and \eqn{stationary_equation_c10=0} -- \eqn{stationary_dist11_2} since $c^{(1+)},c^{(2+)} > 0$ under the condition (a1).
}
\end{remark}

\begin{remark}
{\rm
From the condition (a3), if $c^{(10)} > 0$ and $c^{(20)} > 0$, then \eqn{stationary_dist1}--\eqn{stationary_dist11} are identical with \eqn{stationary_equation_c10=0}--\eqn{stationary_dist11_2}.
}
\end{remark}

\begin{remark}
\label{rem:singular version}
{\rm
We can obtain the reversibility condition even if the condition (ii) is not satisfied. Then, the condition (a4) is slightly changed. Such an example is given in \app{singular_random_walk}. 
}

\end{remark}

We now prove \thr{reversibility condition}.
We first verify that the conditions (a1)--(a4) are necessity for structure-reversibility. 

\begin{lemma}
\label{lem:sta_geo}
{\rm
Under the conditions (i), (ii) and (iii), if $\{\tilde{\vc{Z}}_{\ell}\}$ is a reflecting random walk, then the stationary distribution satisfies \eqn{stationary_dist3}.
}
\end{lemma}
\begin{proof}
From \lem{ran_irr}, we divide our proof into the following three cases.
\begin{eqnarray}
\label{eqn:(D1)}
&&\mbox{$p_{i0}^{(+)} > 0$ and $p_{0j}^{(+)} > 0$ for some $i,j = \pm1$.}\\
\label{eqn:(D2)} 
&&\mbox{$p_{i0}^{(+)} > 0$ for some $i =\pm1$ and $p_{0j}^{(+)} = 0$ for all $j=\pm1$.}\\
\label{eqn:(D3)}
&&\mbox{$p_{0i}^{(+)} > 0$ for some $i =\pm1$ and $p_{j0}^{(+)} = 0$ for all $j=\pm1$.}
\end{eqnarray}
The cases \eqn{(D2)} and \eqn{(D3)} are symmetric, and therefore, we prove  \eqn{stationary_dist3} for \eqn{(D1)} and \eqn{(D2)}.

For the case \eqn{(D1)}, we only consider the case that $p_{10}^{(+)} > 0$ and $p_{01}^{(+)} > 0$ since the other cases are similarly proved.
%
%
Consider the transition of the reversed process $\{\tilde{\vc{Z}}_{\ell}\}$ from $\vc{n}$ to $\vc{n}' = \vc{n} - \vcn{e}_{1}$, where $\vcn{e}_{1} = (1,0)$. Then, it follows from \eqn{reflecting_random_walk} and \eqn{transition_of_reversed} for all $\vc{n} - \vcn{e}_{1} \in \sr{S}_{+}$ (also $\vc{n} \in \sr{S}_{+}$) that 
\begin{eqnarray}
\label{eqn:tra_pro}
\mathbb{P}(\tilde{\vc{Z}}_{\ell + 1} = \vc{n} - \vcn{e}_{1}| \tilde{\vc{Z}}_{\ell} = \vc{n}) = \frac{\pi(\vc{n} - \vcn{e}_{1})}{\pi(\vc{n})} \mathbb{P}(\vc{X}^{(+)} = (1,0)) = \frac{\pi(\vc{n} - \vcn{e}_{1})}{\pi(\vc{n})}p_{10}^{(+)}.
\end{eqnarray}
Since $p_{10}^{(+)} > 0$ and $\{\tilde{\vc{Z}}_{\ell}\}$ is a reflecting random walk, the right hand side must be a constant for all $\vc{n}, \vc{n}' \in \sr{S}_{+}$. Denote this constant by $\eta_{1}^{-1}$, i.e.,
\begin{eqnarray}
\label{eqn:geo_decay}
\eta_{1}^{-1} = \frac{\pi(\vc{n} - \vcn{e}_{1})}{\pi(\vc{n})}, \quad \vc{n},\vc{n} - \vcn{e}_{1} \in \sr{S}_{+}.
\end{eqnarray}
Similarly, we can show that  $\frac{\pi(\vc{n} - \vcn{e}_{2})}{ \pi(\vc{n})}$ is a positive constant for all $\vc{n}, \vc{n} - \vcn{e}_{2} \in \sr{S}_{+}$, where $\vcn{e}_{2} = (0,1)$, and denote it by $\eta_{2}^{-1}$. Thus, for $\vc{n} = (n_{1},n_{2}) \in \sr{S}_{+}$, we have
\begin{eqnarray*}
\pi(\vc{n}) = \eta_{1} \pi(\vc{n} - \vcn{e}_{1}) = \eta_{1}^{n_{1} -1} \pi(1,n_{2}) = \eta_{1}^{n_{1} -1} \eta_{2}^{n_{2} -1} \pi(1,1).
\end{eqnarray*}
Obviously, from this equation, we have $\eta_{1},\eta_{2} \in (0,1)$ since $\sum_{\vc{n} \in \sr{S}_{+}} \pi(\vc{n}) \le 1$. 

We next assume the case \eqn{(D2)}.
Since $p_{i0}^{(+)} > 0$ for some $i \in \{-1,1\}$, we also have \eqn{geo_decay} for the case \eqn{(D2)}.
In what follows,  we prove that $\frac{\pi(\vc{n} - \vcn{e}_{2})}{ \pi(\vc{n})}$ also must be  constant for the case \eqn{(D2)}.

First assume both $p_{10}^{(+)}>0$ and $ p_{(-1)0}^{(+)} > 0$.
Under the condition (ii) and the case \eqn{(D2)} with $p_{10}^{(+)}>0$ and $p_{(-1)0}^{(+)} > 0$, it is clear that the conditions $p_{i(-1)}^{(+)}>0$ and $p_{j1}^{(+)} > 0$ hold for some $i,j \in \{-1,1\}$, and we may assume $p_{(-1)1}^{(+)} > 0$ since a proof is similar to the other cases. Then, it again follows from \eqn{reflecting_random_walk} and \eqn{transition_of_reversed} for $\vc{n} = (n_{1},n_{2})\in \sr{S}_{+}$ and $\vc{n}' = \vc{n} + \vcn{e}_{1} -\vcn{e}_{2}$ , we have 
\begin{eqnarray*}
\mathbb{P}(\tilde{\vc{Z}}_{\ell + 1} = \vc{n} + \vcn{e}_{1} -\vcn{e}_{2}| \tilde{\vc{Z}}_{\ell} = \vc{n}) &=& \frac{\pi(\vc{n} + \vcn{e}_{1} - \vcn{e}_{2})}{\pi(\vc{n})} p_{(-1)1}^{(+)}\\
&=& \frac{\pi(\vc{n} + \vcn{e}_{1} - \vcn{e}_{2})}{\pi(\vc{n})} \frac{\pi(\vc{n} + \vcn{e}_{1})}{\pi(\vc{n} +\vcn{e}_{1})} p_{(-1)1}^{(+)}\\
&=& \frac{\pi(\vc{n} + \vcn{e}_{1} - \vcn{e}_{2})}{\pi(\vc{n} + \vcn{e}_{1})} \eta_{1}p_{(-1)1}^{(+)}\\
&=& \frac{\pi(\vc{m} - \vcn{e}_{2})}{\pi(\vc{m})} \eta_{1}p_{(-1)1}^{(+)}, \quad \vc{m},\vc{m} -\vcn{e}_{2} \in \sr{S}_{+},
\end{eqnarray*}
where $\vc{m} \equiv \vc{n} + \vcn{e}_{1}$ and the third equation is obtained by \eqn{geo_decay}.
Since $p_{(-1)1}^{(+)} > 0$ and the left hand side must be independent for $\vc{m}$, we have
\begin{eqnarray*}
\frac{\pi(\vc{m} - \vcn{e}_{2})}{\pi(\vc{m})} = \eta_{2}^{-1},
\end{eqnarray*}
as long as $\vc{m} \in \sr{S}_{+}$. Thus, we have \eqn{stationary_dist3} for case \eqn{(D2)} with $p_{10}^{(+)} > 0$ and $p_{(-1)0}^{(+)} > 0$.

The rest of proof for \eqn{(D2)} is  the cases that $p_{10}^{(+)} > 0$, $p_{(-1)0}^{(+)} = 0$ and $p_{10}^{(+)} = 0$,  $p_{(-1)0}^{(+)} > 0$. These are also symmetric, so we only consider the case $p_{10}^{(+)} > 0$ and $p_{(-1)0}^{(+)} = 0$. Repeatedly, from the condition (ii), we must have $p_{(-1)1}^{(+)} > 0$ and $p_{(-1)(-1)}^{(+)} > 0$ under the conditions $p_{10}^{(+)} > 0$ and $p_{(-1)0}^{(+)} = 0$. Thus, using the same argument of the case for $p_{10}^{(+)}>0$ and $ p_{(-1)0}^{(+)} > 0$, we obtain \eqn{stationary_dist3}.
\end{proof}

\begin{lemma}
\label{lem:sta_geo_bound}
{\rm
Under the same assumptions of \thr{reversibility condition}, if $\{\tilde{\vc{Z}}_{\ell}\}$ is a reflecting random walk, then we have \eqn{stationary_dist2_1} and \eqn{stationary_dist2_2}.
}
\end{lemma}
\begin{proof}
We only obtain  \eqn{stationary_dist2_1} since \eqn{stationary_dist2_2} is similarly proved. We separately consider the cases such that 
\begin{eqnarray}
\label{eqn:(E1)} 
&&\mbox{Either $p_{10}^{(1)} > 0$ or $p_{(-1)0}^{(1)} > 0$.}\\
\label{eqn:(E2)}
&&\mbox{Both $p_{10}^{(1)} = 0$ and $p_{(-1)0}^{(1)} = 0$.} 
\end{eqnarray}
For \eqn{(E1)}, to change  the proof of the case \eqn{(D1)} in \lem{sta_geo}  from  $\sr{S}_{+}$ into $\sr{S}_{1}$, we have, for some constant $\alpha_{1} \in (0,1)$,
\begin{eqnarray}
\label{eqn:geo_decay_eta}
\alpha_{1}^{-1} = \frac{\pi(\vc{n} - \vcn{e}_{1})}{\pi(\vc{n})}, \quad \vc{n},\vc{n} - \vcn{e}_{1} \in \sr{S}_{1}.
\end{eqnarray}
From the condition (ii), \eqn{geo_decay} and \eqn{geo_decay_eta}, $p_{i(-1)}^{(+)} > 0$ for some $i = 0,\pm1$, and therefore, for $\vc{n} = (n_{1},0) \in \sr{S}_{1}$ and $\vc{n}' = (n_{1}-i,1) \in \sr{S}_{+}$, 
\begin{eqnarray}
\label{eqn:tra_S1_and_S}
\mathbb{P}(\tilde{\vc{Z}}_{\ell + 1} = (n_{1} -i,1)|\tilde{\vc{Z}}_{\ell} = (n_{1},0)) 
= \frac{\pi(n_{1} - i,1)}{\pi(n_{1},0)} p_{i(-1)}^{(+)} = \frac{\eta_{1}^{n_{1} -1-i}\pi(1,1)}{\alpha_{1}^{n_{1}-1} \pi(1,0)} p_{i(-1)}^{(+)}.
\end{eqnarray}
The left hand side of this equation is independent of $n_{1}$ since $\{\tilde{\vc{Z}}_{\ell}\}$ is a reflecting random walk, and therefore,  we have $\alpha_{1} = \eta_{1}$, which implies \eqn{stationary_dist2_1}.

For \eqn{(E2)}, from the irreducible condition (i), 
$p_{i1}^{(1)} > 0$ and $p_{j(-1)}^{(+)} > 0$ for some $i,j = 0,\pm1$. 
Then, from \eqn{reflecting_random_walk}, \eqn{transition_of_reversed} and \eqn{RRRW_equation},  we have, for  $\vc{n} = (n_{1}+ j,0) \in \sr{S}_{1}$ and  $\vc{n}' = (n_{1},1) \in \sr{S}_{+}$,
\begin{eqnarray}
\label{eqn:tilde_transition_+}
\mathbb{P}(\tilde{\vc{Z}}_{\ell + 1} = (n_{1},1) | \tilde{\vc{Z}}_{\ell} = (n_{1}+j,0) ) = \frac{\pi(n_{1},1)}{\pi(n_{1} + j,0)}  p_{j(-1)}^{(+)}.
\end{eqnarray}
Since $\{\vc{Z}_{\ell}\}$ is structure-reversible and $p_{j(-1)}^{(+)} > 0$, the left hand side of this equation is a positive constant which depends on $j$, and denote it by $\tilde{p}_{(-j)1}^{(1)}$. Similarly, for $\vc{n} = (n_{1}+i,1) \in \sr{S}_{+}$ and  $\vc{n}' = (n_{1},0) \in \sr{S}_{1}$,
\begin{eqnarray*}
\mathbb{P}(\tilde{\vc{Z}}_{\ell + 1} = (n_{1},0) | \tilde{\vc{Z}}_{\ell} = (n_{1}+i,1)) 
&=& \frac{\pi(n_{1},0)}{\pi(n_{1}+i,1)} p_{i1}^{(1)}\\
&=& \frac{\pi(n_{1},0)}{\pi(n_{1}+i,1)} \frac{\pi(n_{1} + i  +j,0)}{\pi(n_{1} + i + j,0)}p_{i1}^{(1)}\\
&=& \frac{\pi(n_{1},0)}{\pi(n_{1} + i + j,0)} \frac{p_{j(-1)}^{(+)}}{\tilde{p}_{(-j)1}^{(1)}}p_{i1}^{(1)},
\end{eqnarray*}
where the third equality is given by \eqn{tilde_transition_+}.
Since the left hand side of this equation also does not depend on $n_{1}$, we directly obtain, for $\eta_{1} \in (0,1)$ satisfying \eqn{geo_decay},
\begin{eqnarray}
\label{eqn:sta_geo_boundary_111}
\frac{\pi(n_{1},0)}{\pi(n_{1}+1,0)} = \eta_{1}^{-1}, \quad n_{1} \ge 1,
\end{eqnarray}
if $|i+j| = 1$ holds, and we have \eqn{stationary_dist2_1}.
On the other hand, from irreducible condition (i), if $|i+j| \neq 1$, then at least one of $p_{10}^{(+)} > 0$ and $p_{(-1)0}^{(+)} > 0$ holds, and  we have, for $k=1$ or $k=-1$, 
\begin{eqnarray*}
\mathbb{P}(\tilde{\vc{Z}}_{\ell + 1} = (n_{1},1) | \tilde{\vc{Z}}_{\ell} = (n_{1}+k,1)) 
&=& \frac{\pi(n_{1},1)}{\pi(n_{1} + k,1)} p_{k0}^{(+)}.
\end{eqnarray*}
From \lem{sta_geo}, the probability of the left hand side of this equation is independent for $n_{1}$, and denote its probability by $\tilde{p}_{(-k)0}^{(+)}$
Thus, we obtain, from \eqn{tilde_transition_+}
\begin{eqnarray*}
&&\mathbb{P}(\tilde{\vc{Z}}_{\ell + 1} = (n_{1},0) | \tilde{\vc{Z}}_{\ell} = (n_{1}+i,1)) \\
&& \qquad = \frac{\pi(n_{1},0)}{\pi(n_{1}+i,1)} p_{i1}^{(1)}\\
&& \qquad = \frac{\pi(n_{1},0)}{\pi(n_{1}+i,1)} \frac{\pi(n_{1} + i + k,1)}{\pi(n_{1} + i + k,1)} \frac{\pi(n_{1} + i + j + k,0)}{\pi(n_{1} + i + j + k,0)}p_{i1}^{(1)} \\
&& \qquad = \frac{\pi(n_{1},0)}{\pi(n_{1} + i + j + k,0)} \frac{p_{k0}^{(+)}}{\tilde{p}_{(-k)0}^{(+)}} \frac{p_{j(-1)}^{(+)}}{\tilde{p}_{(-j)1}^{(1)}}p_{i1}^{(1)}.
\end{eqnarray*}
From the irreducible condition (i) again, if $|i+j| \neq 1$, we must have $|i+j+k| = 1$, and therefore, we also have \eqn{sta_geo_boundary_111}. We complete the proof since \eqn{sta_geo_boundary_111} implies \eqn{stationary_dist2_1}.
\end{proof}

\begin{lemma}
{\rm
\label{lem:con_c1}
Under the conditions in \thr{reversibility condition}, if $\{\tilde{\vc{Z}}_{\ell}\}$ is a reflecting random walk, then (a1), (a2) and (a3) hold, and the stationary distribution satisfies \eqn{stationary_dist1}--\eqn{stationary_dist11_2}.
}
\end{lemma}

The proof of \lem{con_c1} is deferred to \app{proof_c1}. We are now ready to prove \thr{reversibility condition}.  

\begin{proof*}{Proof of \thr{reversibility condition}}
From Lemmas \lemt{sta_geo}, \lemt{sta_geo_bound} and \lemt{con_c1}, we already prove necessity of \thr{reversibility condition}. For sufficiency, suppose the conditions (a1)--(a4) hold. Then, from \eqn{reflecting_random_walk} and \eqn{transition_of_reversed}, we have, for $i=0,1,2,+$, $j,k=0,\pm 1$, $\vc{n} \in \sr{S}_{i}$ and $\vc{n}' = \vc{n} + (j,k) \in \sr{S}_{i}$, 
\begin{eqnarray}
\label{eqn:homogeneous_downward}
\mathbb{P}(\tilde{\vc{Z}}_{\ell + 1} = \vc{n}' | \tilde{\vc{Z}}_{\ell} = \vc{n}) &=& \frac{\pi(\vc{n}')}{\pi(\vc{n})} \mathbb{P}(\vc{X}^{(i)} = \vc{n} - \vc{n}') \nonumber\\
&=& \frac{\pi((\vc{n} + (j,k))}{\pi(\vc{n})} \mathbb{P}(\vc{X}^{(i)} = (-j,-k)) \nonumber\\
&=& \eta_{1}^{j} \eta_{2}^{k} p_{(-j)(-k)}^{(i)}.
\end{eqnarray}
Thus,  the transition probabilities of $\{\tilde{\vc{Z}}_{\ell}\}$ from $\sr{S}_{i}$ to $\sr{S}_{i}$ are homogeneous.

We next verify that the downward transitions from ${\sr S}_{+}$ to ${\sr S}_{1}$ and from ${\sr S}_{+}$ to ${\sr S}_{+}$ are homogeneous. 
To this end, let consider the transitions from $\sr{S}_{+}$ to $\sr{S}_{1}$.
For $\vc{n} = (n_{1},1) \in \sr{S}_{+}$ and $\vc{n}' = (n_{1} + i,0) \in \sr{S}_{1}$ and $i=0,\pm1$, we have, by the conditions (a1), (a2), (a3) and (a4)
\begin{eqnarray*}
\mathbb{P}(\tilde{\vc{Z}}_{\ell + 1} = (n_{1} + i,0) | \tilde{\vc{Z}}_{\ell} = (n_{1},1)) 
&=& \frac{\pi(n_{1} + i,0)}{\pi(n_{1},1)} \mathbb{P}(\vc{X}^{(1)} = (-i,1))\\
&=& \frac{\eta_{1}^{i} \pi(1,0)}{\pi(1,1)} p_{(-i)1}^{(1)} \\
&=& \frac{1}{c^{(1+)}}\eta_{1}^{i} \eta_{2}^{-1} p_{(-i)1}^{(1)}\\
&=& \eta_{1}^{i} \eta_{2}^{-1} p_{(-i)1}^{(+)},\\
&=& \mathbb{P}(\tilde{\vc{Z}}_{\ell + 1} = (n_{1} + i,m-1) | \tilde{\vc{Z}}_{\ell} = (n_{1},m)), \quad m \ge 2.
\end{eqnarray*}
where the last equation is obtained by \eqn{homogeneous_downward}.
Note that $(n_{1},m), (n_{1} + i,m-1) \in {\sr S}_{+}$ since we assume $n_{1},n_{1} + i \ge 1$ .
Thus,  the downward transitions for the direction of $n_{1}$-axis in the interior are homogeneous. Using a similar argument, we can prove that another downward transitions are also homogeneous. That is, $\{\tilde{\vc{Z}}_{\ell}\}$ has the homogeneous transitions in each subset $\sr{S}_{i}$. Hence, $\{\tilde{\vc{Z}}_{\ell}\}$ is a reflecting  random walk if the conditions (a1)--(a4) hold. We complete the proof.
\end{proof*}

We discuss about the relationship between structure-reversibility and product form stationary distribution.

\begin{corollary}
\label{cor:reversibility_product form}
{\rm Suppose that the reflecting random walk $\{\vc{Z}_{\ell}\}$ is structure-reversible. Then the stationary distribution of $\{\vc{Z}_{\ell}\}$ has a product form solution if and only if the following conditions are satisfied.
\begin{eqnarray}
\label{eqn:Condition_product1}
&&c^{(1+)} = c^{(20)} \mbox{ if } c^{(20)} > 0, \\
\label{eqn:Condition_product2}
&&c^{(2+)} = c^{(10)} \mbox{ if } c^{(10)} > 0. 
\end{eqnarray}
}
\end{corollary}

The proof of this corollary is deferred to \app{product_form}.

\section{Geometric conditions for structure-reversibility}
\label{sec:Geometric_conditions}
\setnewcounter
We characterize the structure-reversibility in \thr{reversibility condition}.
However, the condition (a4) may not be easily checked because it uses the stationary distribution. Thus,
we replace it by conditions only using modeling primitives.
To this end, we introduce the following notations.
\begin{eqnarray*}
&&\gamma_{0}(\vc{z}) = p_{00}^{(0)} + \frac{c^{(0)}}{c^{(1+)}} p_{(-1)0}^{(1)} z_{1}^{-1} + \frac{c^{(0)}}{{c}^{(2+)}} p_{0(-1)}^{(2)}z_{2}^{-1} + {c}^{(0)} p_{(-1)(-1)}^{(+)} z_{1}^{-1} z_{2}^{-1},\\
&&\gamma_{1}(\vc{z}) = p_{00}^{(1)} + p_{10}^{(1)} z_{1} + p_{(-1)0}^{(1)} z_{1}^{-1}+ {c}^{(1+)} \left(p_{1(-1)}^{(+)} z_{1}z_{2}^{-1} + p_{0(-1)}^{(+)}z_{2}^{-1}  + p_{(-1)(-1)}^{(+)} z_{1}^{-1} z_{2}^{-1} \right),\\
&&\gamma_{2}(\vc{z}) = p_{00}^{(2)} + p_{01}^{(2)} z_{2} + p_{0(-1)}^{(2)} z_{2}^{-1} +{c}^{(2+)}\left( p_{(-1)1}^{(+)} z_{1}^{-1} z_{2} + p_{(-1)0}^{(+)} z_{1}^{-1} + p_{(-1)(-1)}^{(+)} z_{1}^{-1}z_{2}^{-1} \right),\\
&&\gamma_{+}(\vc{z}) = \sum_{i=-1}^{1} \sum_{j=-1}^{1} p_{ij}^{(+)} z_{1}^{i} z_{2}^{j}, 
\end{eqnarray*}
for $\vc{z} \in \dd{R}^{2}$ such that $\vc{z} > (0,0)$, where $c^{(k+)}$ is nonnegative constants defined in \thr{reversibility condition} for $k=1,2$, and 
\begin{eqnarray*}
&&c^{(0)} \equiv \max(c^{(10)}c^{(1+)}, c^{(20)}c^{(2+)}).
\end{eqnarray*}
Note that $\gamma_{+}$ is the generating function of $\vc{X}^{(+)}$. 
Furthermore, if the condition (a3) and $c^{(10)} > 0$ hold, then we have
\begin{eqnarray}
\label{eqn:remark_c01}
c^{(0)} = c^{(10)}c^{(1+)}, \quad \frac{c^{(0)}}{c^{(1+)}} = \frac{c^{(10)}c^{(1+)}}{c^{(1+)}}= c^{(10)}, 
\end{eqnarray}
and if $c^{(20)} > 0$, 
\begin{eqnarray}
\label{eqn:remark_c02}
c^{(0)} = c^{(20)}c^{(2+)}, \quad \frac{c^{(0)}}{c^{(2+)}} = \frac{c^{(20)}c^{(2+)}}{c^{(2+)}} = c^{(20)}.
\end{eqnarray}
In addition, if $c^{(10)} > 0$ and $c^{(20)} > 0$, then we have $c^{(0)} = c^{(10)}c^{(1+)} = c^{(20)}c^{(2+)}$.
We prepare the next lemma, which will be proved at the end of this section.
\begin{lemma}
\label{lem:modeling_primitive}
{\rm
Suppose that conditions (i), (ii) and (a1)--(a3) hold. Then, the condition (iii) and (a4) are equivalent to 

\begin{mylist}{0}
\item[(a5)] There exist $\eta_{1},\eta_{2} \in (0,1)$ satisfying $\gamma_{i}(\eta_{1}^{-1},\eta_{2}^{-1}) = 1$ for all $i=0,1,2,+$.
\end{mylist}
}
\end{lemma}

The next theorem is immediate from this lemma and \thr{reversibility condition}, which gives the conditions of structure-reversibility in terms of modeling parameters. Thus, the theorem is a main result of this paper. 

\begin{theorem}
\label{thr:con_weak_parmeter}
{\rm
Suppose the conditions (i) and (ii). Then, $\{\vc{Z}_{\ell}\}$ is structure-reversible if and only if the conditions (a1)--(a3) and (a5) hold.
}
\end{theorem}

\begin{remark}
{\rm 
(a5) implies the condition (iii). So, it is more convenient than (a4), which needs (iii).
}
\end{remark}

\begin{remark}
\label{rem:LatoMiya2013}
{\rm
Latouche and Miyazawa \cite{LatoMiya2013} derive necessary and sufficient conditions for the stationary distribution of the two dimensional reflecting random walk to have a product form solution (see Theorem 3.9 of \cite{LatoMiya2013}). Similarly to the condition (a5), their conditions have geometric interruptions. We can see that the equations,
$$\gamma_{+}(\eta_{1}^{-1},\eta_{2}^{-1}) = 1, \qquad \gamma_{1}(\eta_{1}^{-1},\eta_{2}^{-1}) = 1, \qquad \gamma_{2}(\eta_{1}^{-1},\eta_{2}^{-1}) = 1$$
 are equivalent to (3.13), (3.14) and (3.17) of \cite{LatoMiya2013}. But it is also notable that the reflecting random walk having a product form stationary distribution may not be structure reversible. Such an example is given in \app{Example of product}. Moreover, the stationary distribution may not have a product form under the structure-reversibility (see \cor{reversibility_product form}). Thus, those two classes of the stationary distributions are slightly different. However, they are identical under some extra conditions. For example, if the conditions:
\begin{eqnarray}
\label{eqn:pro_same1}
&&p_{11}^{(+)} = p_{11}^{(0)} = p_{11}^{(1)} =  p_{11}^{(2)},\\
\label{eqn:pro_same2}
 &&p_{10}^{(0)} = p_{10}^{(1)}, \quad p_{01}^{(0)} = p_{01}^{(2)}, \quad 
p_{1j}^{(2)} = p_{1j}^{(+)}, \quad p_{j1}^{(1)} = p_{j1}^{(+)}, \quad j=0,-1.
\end{eqnarray}
are satisfied, then (a1), (a2) and (a3) hold. Thus, if (a5) holds, then this reflecting random walk is structure reversible. Moreover, 
we have $c^{(10)} = c^{(20)} = 1$, and therefore, $\gamma_{0}(\eta_{1}^{-1},\eta_{2}^{-1}) = 1$ is equivalent to (3.36) of \cite{LatoMiya2013}.
Thus, its stationary distribution has a product form solution. Namely, under the conditions \eqn{pro_same1} and \eqn{pro_same2},  the condition (a5) is just identical with the product form condition (see Theorem 3.12 in \cite{LatoMiya2013}).
}
\end{remark}

We briefly explain the condition (a5).
Obviously, $\gamma_{+}(1,1) = 1$.
Moreover, the subset $\Gamma_{i} = \{\vc{z} \in \mathbb{R}^{2}; \gamma_{i}(\vc{z}) = 1\}$ is a nonnegative-directed convex (see, e.g., \cite{KobaMiya2012}). This means that $\Gamma_{i}$ describe a locally convex curve  on the nonnegative quadrant for $i=0,\pm1, +$. 
In \fig{(a1) and (a2)}, we illustrate the curve $\Gamma_{i}$ that the condition (a5) holds.
\begin{figure}[htbp]
	\centering
	\includegraphics[height=0.25\textheight]{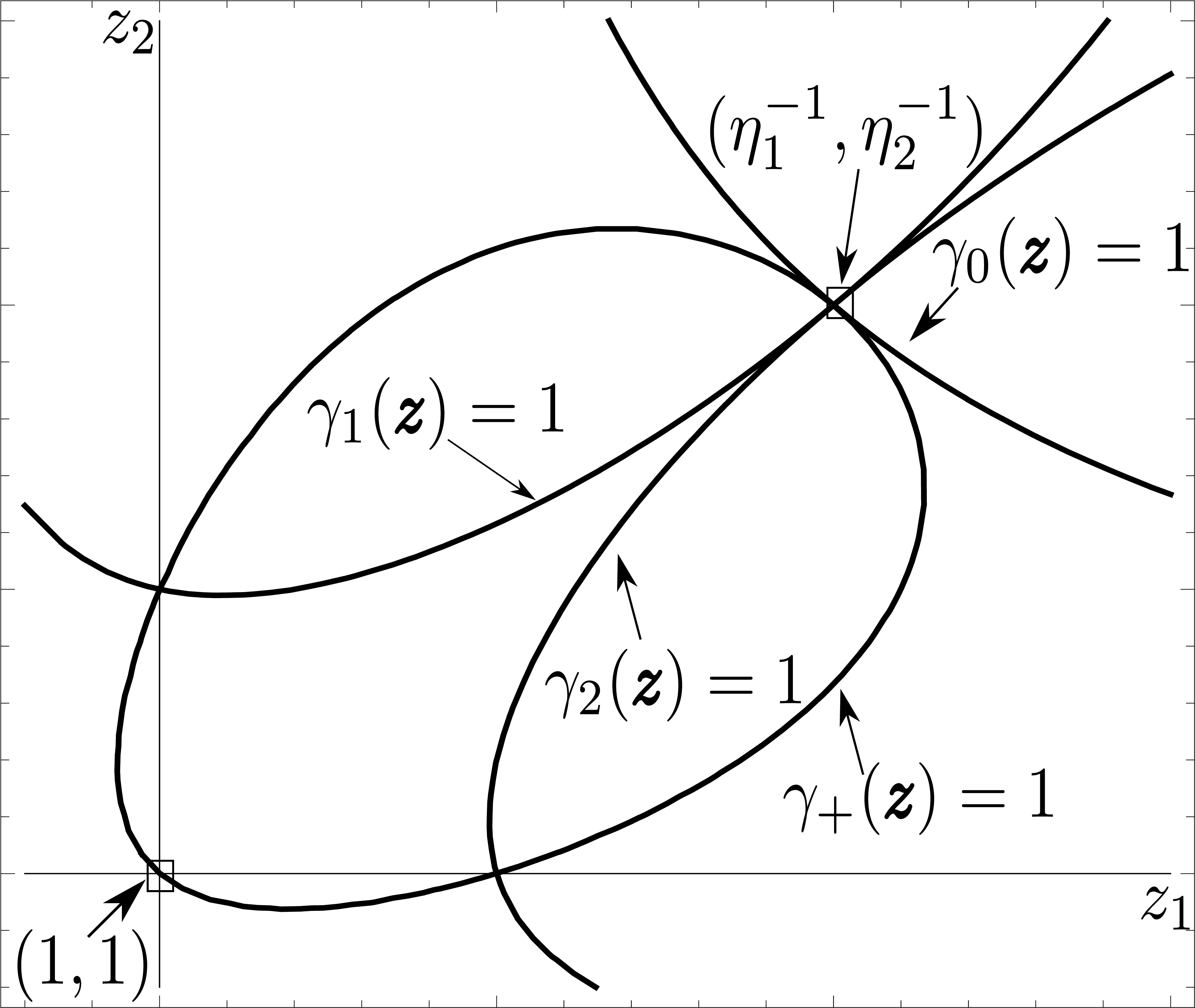}
	\caption{Example curves of $\Gamma_{i}$}
	\label{fig:(a1) and (a2)}
\end{figure}

We are now ready to prove \lem{modeling_primitive}. 

\begin{proof*}{Proof of \lem{modeling_primitive}}
We first prove the necessity of (a5). 
From the condition (iii), we have the following stationary equations. 
\begin{align}
\label{eqn:stationary_equation_0}
&\pi(0,0) = p_{00}^{(0)}\pi(0,0) + p_{(-1)0}^{(1)}\pi(1,0) + p_{0(-1)}^{(2)} \pi(0,1)+ p_{(-1)(-1)}^{(+)} \pi(1,1),\\
\label{eqn:stationary_equation_1}
&\pi(n,0) = \sum_{i=-1}^{1} \left(p_{i0}^{(1)} \pi(n-i,0) + p_{i(-1)}^{(+)} \pi(n-i,1) \right), &&n \ge 2,\\
\label{eqn:stationary_equation_2}
&\pi(0,n) = \sum_{j=-1}^{1} \left(p_{0j}^{(2)} \pi(0,n-j) + p_{(-1)j}^{(+)} \pi(1,n-j) \right), && n \ge 2,\\
\label{eqn:stationary_equation_+}
& \pi(n_{1},n_{2}) = \sum_{i=-1}^{1} \sum_{j=-1}^{1} p_{ij}^{(+)} \pi(n_{1} -i,n_{2} -j), && n_{1},n_{2} \ge 2.
\end{align}
Substituting \eqn{stationary_dist3} into \eqn{stationary_equation_+}, we have $\gamma_{+}(\eta_{1}^{-1},\eta_{2}^{-1}) = 1$. If $c^{(10)} > 0$, from \eqn{stationary_dist2_1}--\eqn{stationary_dist11}, \eqn{stationary_equation_1} and \eqn{stationary_equation_2}, 
\begin{eqnarray*}
&&1 = p_{00}^{(1)} + p_{10}^{(1)} \eta_{1}^{-1} + p_{(-1)0}^{(1)} \eta_{1}+ {c}^{(1+)} \left(p_{1(-1)}^{(+)} \eta_{1}^{-1}\eta_{2}^{} + p_{0(-1)}^{(+)}\eta_{2}  + p_{(-1)(-1)}^{(+)} \eta_{1} \eta_{2}\right),\\
&&\frac{c^{(10)}c^{(1+)}}{c^{(2+)}} = \frac{c^{(10)}c^{(1+)}}{c^{(2+)}}\left(p_{00}^{(2)} + p_{01}^{(2)} \eta_{2}^{-1} + p_{0(-1)}^{(2)} \eta_{2}\right)\\
&& \qquad \qquad \qquad \qquad +c^{(10)}c^{(1+)}\left( p_{(-1)1}^{(+)} \eta_{1} \eta_{2}^{-1} + p_{(-1)0}^{(+)} \eta_{1} + p_{(-1)(-1)}^{(+)} \eta_{1}\eta_{2} \right),\\
\end{eqnarray*}
and we also have $\gamma_{1}(\eta_{1}^{-1},\eta_{2}^{-1}) = 1$ and $\gamma_{2}(\eta_{1}^{-1},\eta_{2}^{-1}) = 1$. If $c^{(20)} > 0$, using \eqn{stationary_equation_c10=0}--\eqn{stationary_dist11_2}, $\gamma_{1}(\eta_{1}^{-1},\eta_{2}^{-1}) = 1$ and $\gamma_{2}(\eta_{1}^{-1},\eta_{2}^{-1}) = 1$ are verified in a similar way. From \rem{irreducibility}, either $c^{(10)} > 0$ or $c^{(20)} > 0$ holds, and therefore, we obtain $\gamma_{1}(\eta_{1}^{-1},\eta_{2}^{-1}) = 1$ and $\gamma_{2}(\eta_{1}^{-1},\eta_{2}^{-1}) = 1$ for all cases.

We finally verify $\gamma_{0}(\eta_{1}^{-1},\eta_{2}^{-1}) = 1$. Using \eqn{stationary_dist2_1}--\eqn{stationary_dist11}, \eqn{remark_c01} and  \eqn{stationary_equation_0}, if $c^{(10)} > 0$, then we have 
\begin{eqnarray*}
1 &=& p_{00}^{(0)} + c^{(10)} p_{(-1)0}^{(1)} \eta_{1}+ \frac{c^{(10)}c^{(1+)}}{c^{(2+)}}p_{0(-1)}^{(2)}\eta_{2} + c^{(10)}c^{(1+)}p_{(-1)(-1)}^{(+)} \eta_{1}\eta_{2} \\
&=& p_{00}^{(0)} + \frac{c^{(0)}}{c^{(1+)}} p_{(-1)0}^{(1)} \eta_{1}+ \frac{c^{(0)}}{c^{(2+)}}p_{0(-1)}^{(2)}\eta_{2} + c^{(0)}p_{(-1)(-1)}^{(+)} \eta_{1}\eta_{2}\\
&=& \gamma_{0}(\eta_{1}^{-1},\eta_{2}^{-1}).
\end{eqnarray*}
The case $c^{(20)} > 0$ is similarly proved, which is completed the proof of the necessity.

We  next verify the sufficiency of (a5). From \rem{irreducibility} again, either $c^{(10)} > 0$ or $c^{(20)}>0$ holds. We only prove \lem{modeling_primitive} for the case $c^{(10)} > 0$ since the case $c^{(20)} > 0$ is similarly proved. Thus, suppose that the conditions (a1)--(a3), (a5) and $c^{(10)} > 0$ hold. Let $\pi$ be a function from ${\sr S}$ to $[0,1]$ satisfying \eqn{stationary_dist2_1}--\eqn{stationary_dist11}. To verify the condition (iii) and (a4), it suffices to prove that this $\pi$ is a finite measure and satisfies the stationary equation, that is, 
\begin{eqnarray}
\label{eqn:check_stationary2}
&&\sum_{\vc{n} \in {\sr S}} \pi(\vc{n}) < \infty,\\
\label{eqn:check_stationary1}
&&\sum_{\vc{n'} \in {\sr S}}\frac{\pi({\vc{n'}})}{\pi(\vc{n})} \mathbb{P}(\vc{Z}_{\ell+1} = \vc{n} | \vc{Z}_{\ell} = \vc{n}') = 1,  \qquad \vc{n} \in {\sr S},
\end{eqnarray}
since $\pi(0,0)$ is arbitrarily given. 
It is easy to verify \eqn{check_stationary2} since $\eta_{1},\eta_{2} \in (0,1)$. To verify \eqn{check_stationary1}, we consider the following six cases.
\begin{eqnarray*}
\vc{n} = (0,0),(1,0),(1,1),(n_{1},0),(n_{1},1),(n_{1},n_{2}), \quad n_{1},n_{2} \ge 2.
\end{eqnarray*}
We omit the other cases that $\vc{n} = (0,1),(0,n_{2}),(1,n_{2})$ because these cases are symmetric to $(1,0),(n_{1},0),(n_{1},1)$, respectively. 
For $\vc{n} = (0,0)$, we have $\vc{n}' \in \{(0,0), (1,0), (0,1), (1,1)\}$, and therefore, 
\begin{eqnarray*}
&&\sum_{\vc{n'} \in {\sr S}}\frac{\pi({\vc{n'}})}{\pi(\vc{n})} \mathbb{P}(\vc{Z}_{\ell+1} = \vc{n} | \vc{Z}_{\ell} = \vc{n}') \\
&& \qquad = p_{00}^{(0)} + \frac{\pi( 1 , 0 )}{\pi( 0 , 0 )} p_{(-1)0}^{(1)} + \frac{\pi( 0 , 1 )}{\pi( 0 , 0 )} p_{0(-1)}^{(2)} + \frac{\pi( 1 , 1 )}{\pi( 0 , 0 )} p_{(-1)(-1)}^{(+)}\\
&& \qquad  = p_{00}^{(0)} + c^{(10)} p_{(-1)0}^{(1)} \eta_{1} + \frac{c^{(10)}c^{(1+)}}{c^{(2+)}} p_{0(-1)}^{(2)} \eta_{2} + c^{(10)}c^{(1+)} p_{(-1)(-1)}^{(+)} \eta_{1} \eta_{2}\\
&& \qquad  = p_{00}^{(0)} + \frac{c^{(1+)}}{c^{(0)}} p_{(-1)0}^{(1)} \eta_{1} + \frac{c^{(0)}}{c^{(2+)}} p_{0(-1)}^{(2)} \eta_{2} + c^{(0)} p_{(-1)(-1)}^{(+)} \eta_{1} \eta_{2}\\
&& \qquad = \gamma_{0} (\eta_{1}^{-1} , \eta_{2}^{-1}) = 1, 
\end{eqnarray*}
where the second and third equalities are obtained by \eqn{stationary_dist1}--\eqn{stationary_dist11} and \eqn{remark_c01}. The other cases are similarly proved, but we give their detailed proofs in \app{check_stationary1} for the reader. This completes the proof.
\end{proof*}

\section{Application to a queueing network}
\setnewcounter
\label{sec:exam}
In this section, we construct a discrete time queueing network whose stationary distribution is not of product form but has closed form, using structure-reversibility.
For this, we modify a discrete time Jackson network, which is introduced below.
We define the reflecting random walk by
\begin{eqnarray*}
&&p_{11}^{(i)} = 0, \quad p_{(-1)(-1)}^{(i)} = 0, \quad  
p_{10}^{(i)} = \lambda_{1}, \quad p_{01}^{(i)} = \lambda_{2}, \quad i=0,1,2,+,\\
&&p_{(-1)0}^{(j)} = \mu_{1} r_{10}, \quad p_{(-1)1}^{(j)} = \mu_{1} r_{12}, \quad j = 1,+,\\
&& p_{0(-1)}^{(k)} = \mu_{2} r_{20}, \quad p_{1(-1)}^{(k)} = \mu_{2} r_{21}, \quad k = 2,+,\\
&& p_{00}^{(+)} = \mu_{1} r_{11} + \mu_{2} r_{22}, \quad p_{11}^{(1)} = \mu_{1} r_{11} + \mu_{2}, \\
&& p_{00}^{(2)} = \mu_{1} + \mu_{2} r_{22}, \quad p_{11}^{(0)} = \mu_{1} + \mu_{2}, 
\end{eqnarray*}
where all constants are positive and  
\begin{eqnarray*}
&&\lambda_{1} + \lambda_{2} + \mu_{1} + \mu_{2} = 1, \\
&&r_{10} + r_{11} + r_{12} = 1,\\
&&r_{20} + r_{21} + r_{22} = 1.
\end{eqnarray*}
This random walk is called  a discrete time Jackson network. 
For $i=1,2$, let $\alpha_{i}$ be a solution of the following traffic equations.
\begin{eqnarray}
\label{eqn:traffic_equation}
\alpha_{1} = \lambda_{1}  + \alpha_{2} r_{21} + \alpha_{1} r_{11}, \quad \alpha_{2} = \lambda_{2}  + \alpha_{1} r_{12} + \alpha_{2} r_{22}.
\end{eqnarray}
The stability condition for this model is given by
\begin{eqnarray*}
\rho_{1} \equiv \frac{\alpha_{1}}{\mu_{1}} < 1, \quad \rho_{2} \equiv \frac{\alpha_{2}}{\mu_{2}} < 1.
\end{eqnarray*}
It is well known that the stationary distribution of the discrete time Jackson network has a product form solution, that is, stationary distribution is given by
\begin{eqnarray*}
\pi(n_{1},n_{2}) = (1 - \rho_{1}) (1- \rho_{2})\rho_{1}^{n_{1}} \rho_{2}^{n_{2}}, \quad n_{1},n_{2} \ge 0.
\end{eqnarray*}
In addition, it is easy to verify the conditions (a1)--(a3) and (a5), and therefore, the discrete time Jackson network is structure-reversible.

For the Jackson network introduced above, let us consider the reflecting random walk whose transition probabilities in the interior of the quadrant are given by
\begin{eqnarray*}
&&p_{11}^{(+)} = 0, \quad p_{(-1)(-1)}^{(+)} = 0, \quad  
p_{10}^{(+)} = \lambda_{1}, \quad p_{01}^{(+)} = \lambda_{2},\\
&&p_{(-1)0}^{(+)} = \mu_{1} r_{10}, \quad p_{(-1)1}^{(+)} = \mu_{1} r_{12}, \quad p_{0(-1)}^{(+)} = \mu_{2} r_{20}, \\
&& p_{1(-1)}^{(+)} = \mu_{2} r_{21}, \quad p_{00}^{(+)} = \mu_{1} r_{11} + \mu_{2} r_{22},
\end{eqnarray*}
and $p^{(0)}_{11} = p^{(1)}_{11} = p^{(2)}_{11} = 0$. Note that the other transition probabilities may be arbitrarily given. We refer to this reflecting random walk as a queueing network flexible at the boundary with no simultaneous movement. For simplicity, we omit ``with no simultaneous movement'' in what follows.

For this queueing network, it is easy to see that the structure reversibility conditions (a1) and (a2) in the  \thr{con_weak_parmeter} are simplified to the following two conditions.
\begin{itemize}
\item[(b1)] $\frac{p_{01}^{(1)}}{p_{(-1)1}^{(1)}} = \frac{\lambda_{2}}{\mu_{1} r_{12}}$ and $\frac{p_{10}^{(2)}}{p_{1(-1)}^{(2)}} = \frac{\lambda_{1}}{\mu_{2} r_{21}}$.
\item[(b2)] $p_{10}^{(1)}=0$ if and only if $p_{10}^{(0)} = 0$. Similarly, $p_{01}^{(2)}=0$ if and only if $p_{01}^{(0)} = 0$. 
\end{itemize}
Namely, the queueing network flexible at the boundary is structure-reversible if and only if the conditions (b1), (b2), (a3) and (a5) are satisfied.

We next consider a special case of the queueing network flexible at the boundary for constructing an example such that it is structure-reversible but does not have a product form stationary distribution.
For this, we put the transition probabilities at the boundary of the queueing network as follows.
\begin{eqnarray*}
&&p_{10}^{(1)} = p_{10}^{(+)} = \lambda_{1}, \quad p_{01}^{(2)} = p_{01}^{(+)} = \lambda_{2}, \quad p_{(-1)0}^{(1)} = p_{(-1)0}^{(+)} = \mu_{1} r_{10}, \quad p_{0(-1)}^{(2)} = p_{0(-1)}^{(+)} = \mu_{2} r_{20}, \\
&&p_{01}^{(1)} = \lambda_{2} + \lambda_{2}^{(1)}, \quad p_{10}^{(2)} = \lambda_{1} + \lambda_{1}^{(2)}, \quad p_{10}^{(0)} = \lambda_{1} + \lambda_{1}^{(0)}, \quad p_{01}^{(0)} = \lambda_{2} + \lambda_{2}^{(0)},\\
&&p_{(-1)1}^{(1)} = \mu_{1} (r_{12} + r_{11}), \quad p_{1(-1)}^{(2)} = \mu_{2} (r_{21}+ r_{22}), \quad p_{00}^{(1)} = \mu_{2} - \lambda_{2}^{(1)},\\
&& p_{00}^{(2)} = \mu_{1} - \lambda_{1}^{(2)}, \quad p_{00}^{(0)} = \mu_{1} + \mu_{2} - \lambda_{1}^{(0)} - \lambda_{2}^{(0)}, 
\end{eqnarray*}
 where we assume $0 \le \lambda_{2}^{(1)} \le \mu_{2}$, $0 \le \lambda_{1}^{(2)} \le \mu_{1}$ and $0 \le \lambda_{1}^{(0)} +  \lambda_{2}^{(0)} \le \mu_{1} + \mu_{2}$.
We refer to this queueing network  as a discrete time Jackson network with extra arrivals at empty nodes. In \fig{transtion_probability_JEE}, we depict the transition diagram of this queueing network.
\begin{figure}[htbp]
	\centering
	\includegraphics[height=0.35\textheight]{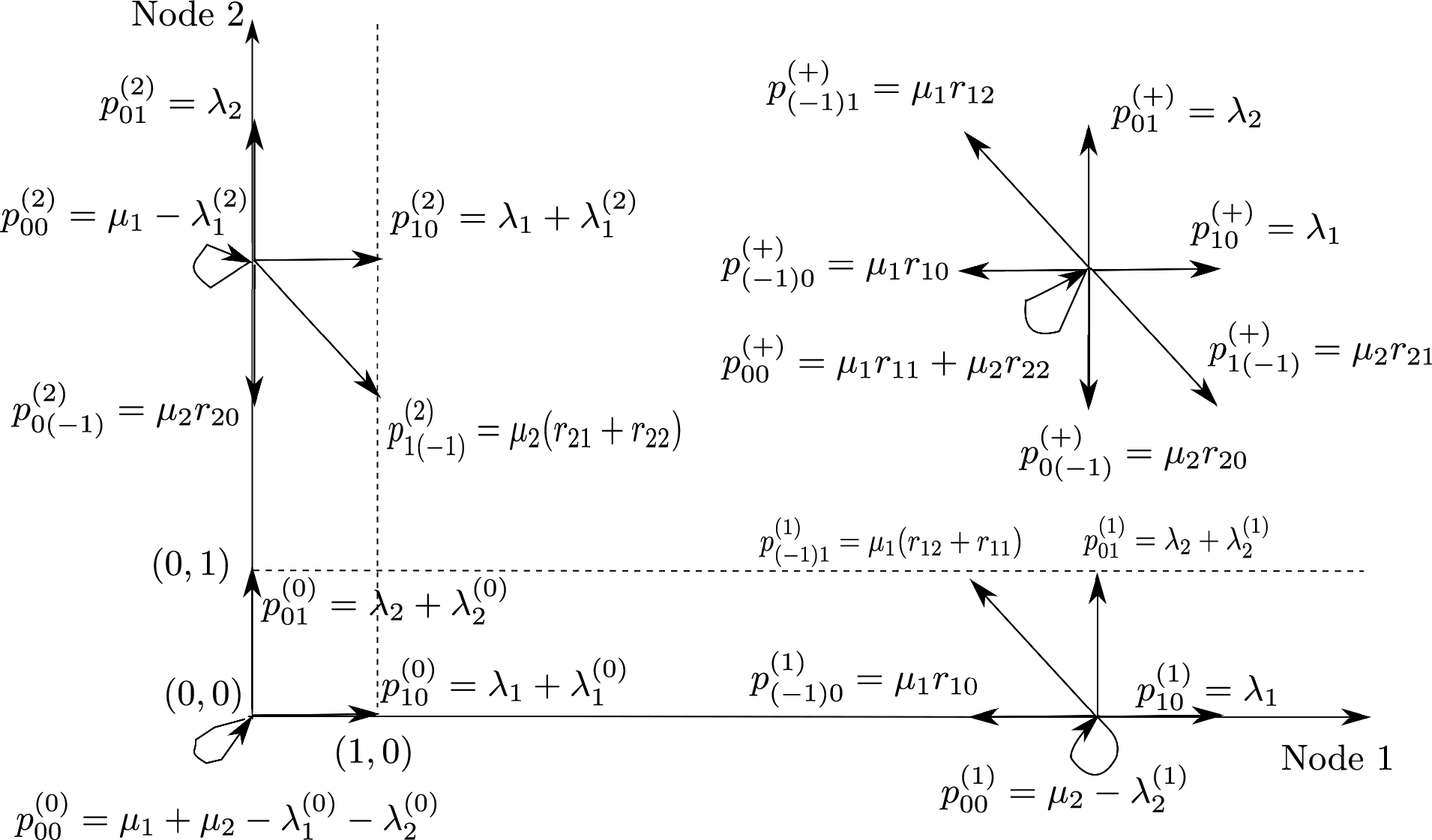}
	\caption{Transition diagram}
	\label{fig:transtion_probability_JEE}
\end{figure}

For the discrete time Jackson network with extra arrivals at empty nodes, it is easy to verify that the condition (b2) is satisfied.  Assume that 
\begin{eqnarray}
\label{eqn:Jackson_assumption}
\frac{\lambda_{2} + \lambda_{2}^{(1)}}{\lambda_{2}}= \frac{r_{12} + r_{11}}{r_{12}}, \quad \frac{\lambda_{1} + \lambda_{1}^{(2)}}{\lambda_{1}} = \frac{r_{21} + r_{22}}{r_{21}}.
\end{eqnarray}
Then, 
\begin{eqnarray*}
\frac{p_{01}^{(1)}}{p_{(-1)1}^{(1)}} = \frac{\lambda_{2} + \lambda_{2}^{(1)}}{\mu_{1} (r_{12} + r_{11})} = \frac{\lambda_{2}}{\mu_{1} r_{12}}, \quad \frac{p_{10}^{(2)}}{p_{1(-1)}^{(2)}} = \frac{\lambda_{1} + \lambda_{1}^{(2)}}{\mu_{2} (r_{21} + r_{22})} = \frac{\lambda_{1}}{\mu_{2} r_{21}}.
\end{eqnarray*}
Hence, the condition (b1) is satisfied . Moreover, we have 
\begin{eqnarray}
\label{eqn:c_10 c_20 c+}
c^{(1+)} = 1 + \frac{\lambda_{2}^{(1)}}{\lambda_{2}}, \quad c^{(2+)} =  1 + \frac{\lambda_{1}^{(2)}}{\lambda_{1}}, \quad c^{(10)} = 1 + \frac{\lambda_{1}^{(0)}}{\lambda_{1}}, \quad  c^{(20)} = 1 + \frac{\lambda_{2}^{(0)}}{\lambda_{2}}. 
\end{eqnarray}
We also assume the condition (a3), that is, $c^{(10)}c^{(1+)} = c^{(20)}c^{(2+)}$ which is equivalent to 
\begin{eqnarray}
\label{eqn:Jackson_assumption2}
\left(1 + \frac{\lambda_{1}^{(0)}}{\lambda_{1}} \right)\left(1 + \frac{\lambda_{2}^{(1)}}{\lambda_{2}} \right) = \left(1 + \frac{\lambda_{2}^{(0)}}{\lambda_{2}} \right)\left(1 + \frac{\lambda_{1}^{(2)}}{\lambda_{1}} \right).
\end{eqnarray}
Moreover, for the condition (a5), we assume the following conditions.
\begin{eqnarray}
\label{eqn:parameter_condition_1}
&&\lambda_{1} = \lambda_{2} \equiv \lambda,\\
\label{eqn:parameter_condition_2}
&&\frac{r_{12}}{r_{10}} = \frac{r_{21}}{r_{20}}. 
\end{eqnarray}
Then,  we confirm that the condition (a5) is satisfied with $\eta_{1} = \rho_{1}$ and $\eta_{2} = \rho_{2}$ (see, \app{computation}). Thus, under the conditions \eqn{Jackson_assumption}--\eqn{parameter_condition_2}, the discrete time Jackson network with extra arrivals at empty nodes is structure-reversible.  

From \eqn{Jackson_assumption}, \eqn{Jackson_assumption2} and \eqn{parameter_condition_1}, we have 
\begin{eqnarray}
\label{eqn:lambda_condition1}
&&\frac{\lambda + \lambda_{1}^{(0)}}{\lambda + \lambda_{2}^{(0)}} = \frac{\lambda  + \lambda_{1}^{(2)}}{\lambda + \lambda_{2}^{(1)}},\\
\label{eqn:lambda_condition2}
&&\lambda_{1}^{(2)} = \frac{r_{22}}{r_{21}}\lambda, \quad \lambda_{1}^{(2)} = \frac{r_{11}}{r_{12}}\lambda.
\end{eqnarray}
These imply that $\lambda_{1}^{(2)}$, $\lambda_{2}^{(1)}$, $\lambda_{1}^{(0)}$ and $\lambda_{2}^{(0)}$ are determined by $\lambda$, $r_{11}$, $r_{12}$, $r_{21}$ and $r_{22}$.
It is easy to see that the condition \eqn{Jackson_assumption2} is satisfied if 
\begin{eqnarray}
\label{eqn:product_form_lambda}
\lambda_{1}^{(2)} = \lambda_{1}^{(0)}, \quad \lambda_{2}^{(1)} = \lambda_{2}^{(0)},
\end{eqnarray}
and we obviously have \eqn{Condition_product1} and \eqn{Condition_product2}.  
Thus, from \cor{reversibility_product form}, under the structure-reversibility conditions \eqn{Jackson_assumption}--\eqn{parameter_condition_2}, the stationary distribution of this network has a product form solution if and only if \eqn{product_form_lambda} is satisfied. 
From \eqn{lambda_condition1} and \eqn{lambda_condition2}, there may be the case where \eqn{product_form_lambda} does not hold while \eqn{Jackson_assumption}--\eqn{parameter_condition_2} are satisfied. We give such an example below.
\begin{align}
&\lambda_{1} = \lambda_{2} = 0.0667, &&\mu_{1} = 0.4000, &&\mu_{2} = 0.4667, &&r_{10} = 0.368, &&r_{11} = r_{12} = 0.3158, \nonumber\\
&r_{20} = 0,3784, &&r_{21} = 0.3243, &&r_{22} = 0.2973, &&\lambda_{2}^{(1)} = 0.0667, &&\lambda_{1}^{(2)} = 0.0611,\nonumber\\
\label{eqn:JN_with_additional_arrivals}
& \lambda_{1}^{(0)} = 0.2104, &&\lambda_{2}^{(0)} = 0.2225.
\end{align}
Thus, the discrete time Jackson network with additional arrivals at empty nodes may not have a product form solution when it is structure-reversible. 
For this example, we also have $\rho_{1}^{-1} = 2.2105$ and $\rho_{2}^{-1} = 2.6486$ (see \fig{Gamma_JEE}).

\begin{figure}[htbp]
	\centering
	\includegraphics[height=0.35\textheight]{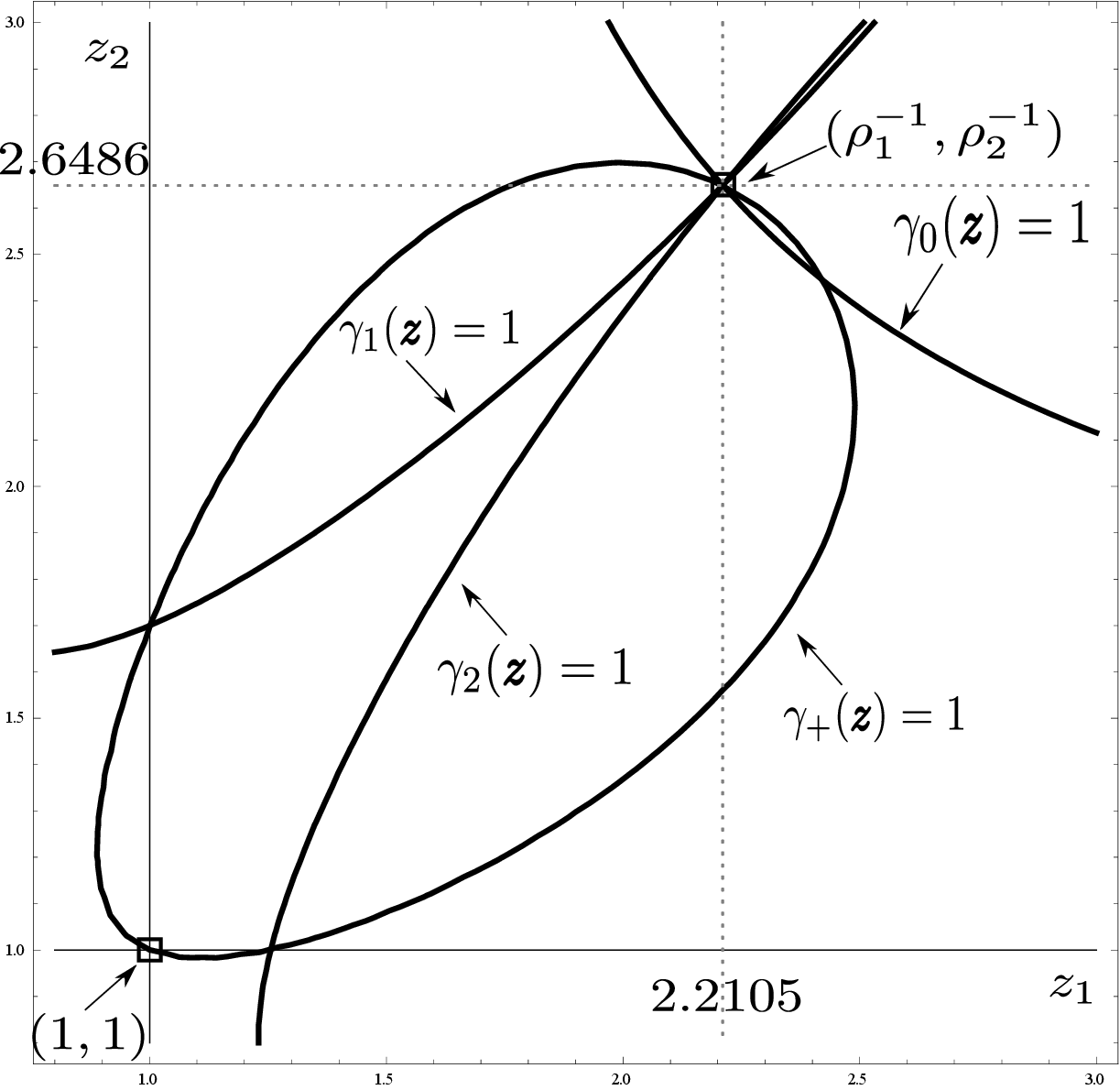}
	\caption{The curves of $\Gamma_{i}$ satisfying  \eqn{JN_with_additional_arrivals}}
	\label{fig:Gamma_JEE}
\end{figure}

\section*{Acknowledgements}
We are grateful to an anonymous referee for its helpful comments. This research was supported in part by Japan Society for the Promotion of Science under grant No.\ 24310115.

\appendix

\section{Singular reflecting random walk}
\setnewcounter
\label{app:singular_random_walk}

In this section, we obtain a structure-reversibility condition in the special case.
For this, we assume the following conditions. 
\begin{eqnarray}
\label{eqn:singular_condition1}
&&p^{(+)}_{ij} = 0, \quad i=0,\pm1, j=\pm1, \\
\label{eqn:singular_condition2}
&&p^{(+)}_{i0} > 0, \quad i=0,\pm1.
\end{eqnarray}
Then, it is easy to see that the random walk $\{\vc{Y}_{\ell}\}$ is not irreducible. 
This reflecting random walk is referred to as a singular reflecting random walk, which is introduced by \cite{FayoIasnMaly1999}. 
From irreducibility condition (i), we must have, for some $i,j=0,1$ 
\begin{eqnarray}
\label{eqn:lineary_condition1}
&&p_{10}^{(1)} > 0, \quad p_{(-1)0}^{(1)} > 0,\\
\label{eqn:lineary_condition2}
&&p_{i1}^{(2)} > 0, \quad p_{j(-1)}^{(2)} > 0.
\end{eqnarray}
In \fig{singular_randomwalk}, we depict transition diagram of the reflecting random walk satisfying the conditions \eqn{singular_condition1}--\eqn{lineary_condition2}.
\begin{figure}[htbp]
 	\centering
	\includegraphics[height=0.25\textheight]{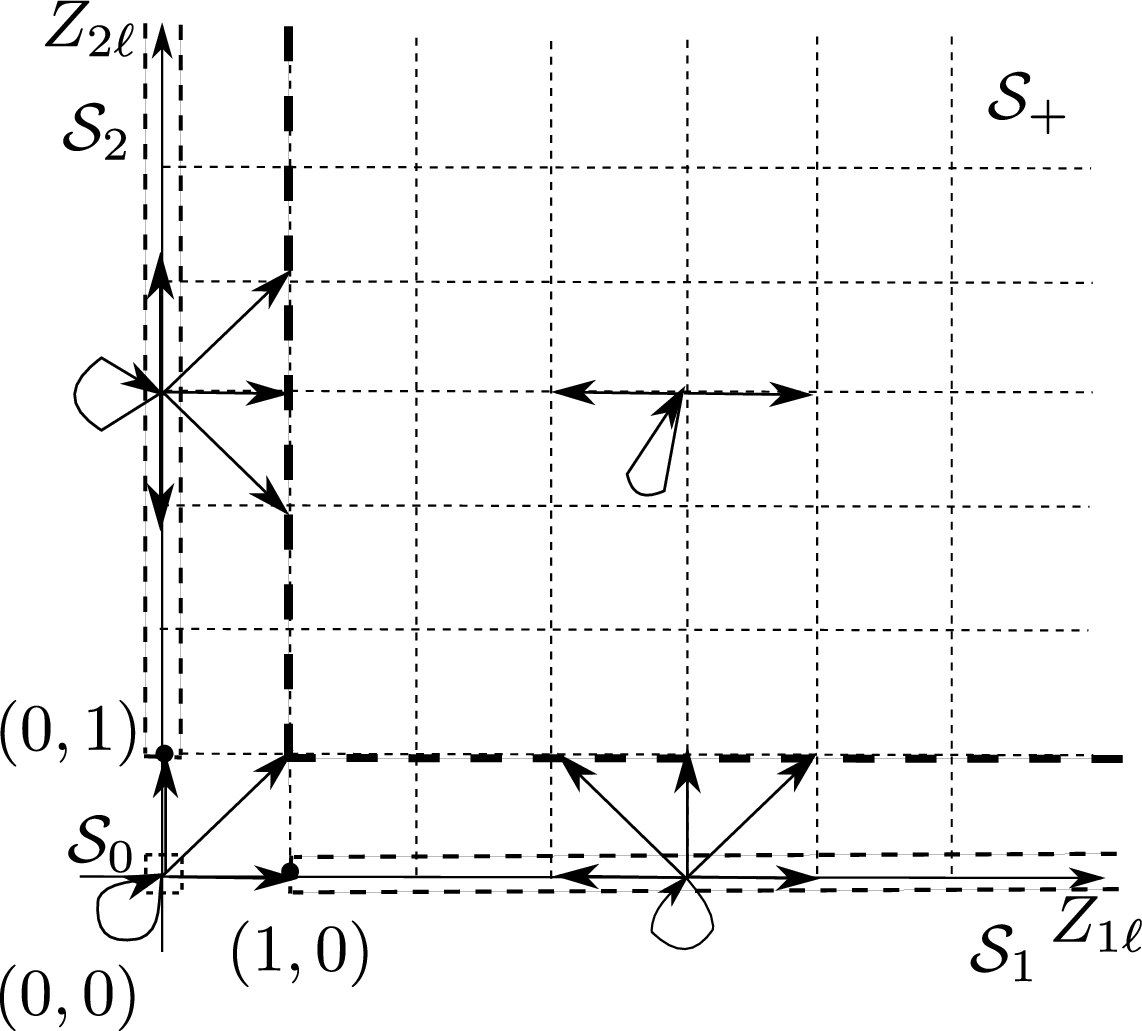}
	\caption{Singular reflecting random walk satisfying the irreducible condition (i)}
	\label{fig:singular_randomwalk}
\end{figure}

This reflecting random walk is structure-reversible if and only if the following conditions hold.
\begin{eqnarray*}
\begin{array}{llll}
&&\pi(n_{1},n_{2}) = \eta_{1}^{n_{1} -1}\eta_{2}^{n_{2} -1} \pi(1,1), &n_{1} \ge 1, n_{2} \ge 1, \\
&&\pi(n_{1},0) = \alpha_{1}^{n_{1} -1} \pi(1,0), &n_{1} \ge 1, \\
&&\pi(0,n_{2}) = \eta_{2}^{n_{2}-1} \pi(0,1), &n_{2} \ge 1, \\
&&\pi(1,0) = c^{(10)} \alpha_{1} \pi(0,0),\\
&&\pi(0,1) = c^{(20)}\eta_{2} \pi(0,0),  \\
&&\pi(1,1) = c^{(20)}c^{(2+)} \eta_{1}\eta_{2} \pi(0,0), \\
&& p_{i1}^{(1)} = 0, \quad p_{1j}^{(2)} = 0, \quad p_{11}^{(0)} = 0,& i=0,\pm1, j = \pm1,
\end{array}
\end{eqnarray*}
where $\eta_{1},\eta_{2},\alpha_{1} \in (0,1)$ and $c^{(2+)}, c^{(10)}, c^{(20)} > 0$ are given by
\begin{eqnarray*}
&&\eta_{1} = \frac{p_{10}^{(+)}}{p_{(-1)0}^{(+)}}, \quad \eta_{2} = \frac{p_{01}^{(2)}}{p_{0(-1)}^{(2)}}, \quad \alpha_{1} = \frac{p_{10}^{(1)}}{p_{(-1)0}^{(1)}},\\
&&c^{(2+)} = \frac{p_{10}^{(2)}}{p_{10}^{(+)}}, \quad c^{(10)} = \frac{p_{10}^{(0)}}{p_{10}^{(1)}}, \quad c^{(20)} = \frac{p_{01}^{(0)}}{p_{01}^{(2)}}.
\end{eqnarray*}
In what follows, we will derive these conditions.

 We assume that $\{\vc{Z}_{\ell}\}$ is structure-reversible. Then, for $i,j = 0,\pm1$ and $(n_{1},n_{2}) \in {\sr S}_{k}$,  we can define the following probability function $\tilde{p}_{ij}^{(k)}$.
\begin{eqnarray*}
\tilde{p}_{ij}^{(k)} = \mathbb{P}(\tilde{\vc{Z}}_{\ell+1} = (n_{1}+i,n_{2}+j) | \tilde{\vc{Z}}_{\ell} = (n_{1},n_{2})).
\end{eqnarray*}
From \eqn{transition_of_reversed}, we have 
\begin{eqnarray*}
&&\tilde{p}_{10}^{(+)} = \mathbb{P}(\tilde{\vc{Z}}_{\ell + 1} = (n_{1} + 1,n_{2}) | \tilde{\vc{Z}}_{\ell} = (n_{1},n_{2})) = \frac{\pi(n_{1}+1,n_{2})}{\pi(n_{1},n_{2})} p_{(-1)0}^{(+)} > 0, \quad (n_{1},n_{2}) \in {\sr S}_{+}.\\
\end{eqnarray*}
Thus, if $\{\vc{Z}_{\ell}\}$ is structure-reversible, then we must have
\begin{eqnarray}
\label{eqn:singular1_stationary1}
\pi(n_{1} + 1,n_{2}) = \eta_{1} \pi(n_{1},n_{2}), \quad (n_{1},n_{2}) \in {\sr S}_{+},
\end{eqnarray}
for some $\eta_{1} \in (0,1)$. 
Similarly,  from \eqn{lineary_condition1}, we have 
\begin{eqnarray*}
&&\tilde{p}_{10}^{(1)} = \mathbb{P}(\tilde{\vc{Z}}_{\ell + 1} = (n_{1} + 1,0) | \tilde{\vc{Z}}_{\ell} = (n_{1},0)) = \frac{\pi(n_{1}+1,0)}{\pi(n_{1},0)} p_{(-1)0}^{(1)} > 0, \quad (n_{1},0) \in {\sr S}_{1}.\\
\end{eqnarray*}
Thus, for some $\alpha_{1} \in (0,1)$, 
\begin{eqnarray}
\label{eqn:singular1_stationary2}
\pi(n_{1} + 1,0) = \alpha_{1}\pi(n_{1},0), \quad (n_{1},0) \in {\sr S}_{1}.
\end{eqnarray}
On the other hand, we have, for $i = 0,\pm1$ and  $(n_{1} - i,n_{2}-1), (n_{1},n_{2}) \in {\sr S}_{+}$, 
\begin{eqnarray*}
\tilde{p}_{(-i)(-1)}^{(+)} = \mathbb{P}(\tilde{\vc{Z}}_{\ell + 1} = (n_{1} -i,n_{2} -1) | \tilde{\vc{Z}}_{\ell} = (n_{1},n_{2})) = \frac{\pi(n_{1}-i,n_{2} - 1)}{ \pi(n_{1},n_{2})}p_{i1}^{(+)} = 0,
\end{eqnarray*}
since we assume \eqn{singular_condition1}. For $n_{1} > 1$, 
\begin{eqnarray*}
\tilde{p}_{(-i)(-1)}^{(+)} = \mathbb{P}(\tilde{\vc{Z}}_{\ell + 1} = (n_{1} -i,0) | \tilde{\vc{Z}}_{\ell} = (n_{1},1) ) = \frac{\pi(n_{1} -i,0)}{\pi(n_{1},1)} p_{i1}^{(1)}.
\end{eqnarray*}
Thus, we must have $p_{i1}^{(1)} = 0$ for any $i=0,\pm1$. Similarly, we have,  for $j=\pm1$ and $(n_{1} - 1,n_{2} - j), (n_{1},n_{2}) \in {\sr S}_{+}$, 
\begin{eqnarray*}
\tilde{p}_{(-1)(-j)}^{(+)} = \mathbb{P}(\tilde{\vc{Z}}_{\ell + 1} = (n_{1} -1,n_{2} -j) | \tilde{\vc{Z}}_{\ell} = (n_{1},n_{2})) = \frac{\pi(n_{1}-1,n_{2} - j)}{ \pi(n_{1},n_{2})}p_{1j}^{(+)} = 0,
\end{eqnarray*}
from \eqn{singular_condition1}.  Since
\begin{eqnarray*}
\tilde{p}_{(-1)(-j)}^{(+)} = \mathbb{P}(\tilde{\vc{Z}}_{\ell + 1} = (0,n_{2} -j) | \tilde{\vc{Z}}_{\ell} = (1,n_{2})) = \frac{\pi(0,n_{2}-j)}{\pi(1,n_{2})} p_{1j}^{(2)}, 
\end{eqnarray*}
we also have $p_{1j}^{(2)} = 0$ for $j=\pm1$, and from \eqn{lineary_condition2}, $p_{01}^{(2)} > 0$ and $p_{0(-1)}^{(2)} > 0$. 
Moreover,  for $n_{2} \ge 1$, 
\begin{eqnarray*}
&&\tilde{p}_{01}^{(2)} = \mathbb{P}(\tilde{\vc{Z}}_{\ell + 1} = (0,n_{2} + 1) | \tilde{\vc{Z}}_{\ell} = (0,n_{2})) = \frac{\pi(0,n_{2} + 1)}{\pi(0,n_{2})} p_{0(-1)}^{(2)} > 0.\\
\end{eqnarray*}
Thus, for $\eta_{2} \in (0,1)$, 
\begin{eqnarray*}
\pi(0,n_{2} + 1) = \eta_{2} \pi(0,n_{2}), \quad (0,n_{2}) \in {\sr S}_{2}.
\end{eqnarray*}
We next have, for any $n_{2} \ge 1$ and $n_{1} \ge 1$,
\begin{eqnarray*}
&&\tilde{p}_{(-1)0}^{(+)} = \mathbb{P}(\tilde{\vc{Z}}_{\ell + 1} = (n_{1},n_{2}) | \tilde{\vc{Z}}_{\ell} = (n_{1} + 1,n_{2})) = \frac{\pi(n_{1},n_{2})}{\pi(n_{1} + 1,n_{2})} p_{10}^{(+)} = \eta_{1}^{-1} p_{10}^{(+)} > 0,\\ 
&&\tilde{p}_{(-1)0}^{(+)} = \mathbb{P}(\tilde{\vc{Z}}_{\ell + 1} = (0,n_{2}) | \tilde{\vc{Z}}_{\ell} = (1,n_{2})) = \frac{\pi(0,n_{2})}{\pi(1,n_{2})} p_{10}^{(2)} = \frac{\eta_{2}^{n_{2}-1}\pi(0,1)}{\pi(1,n_{2})} p_{10}^{(2)} = \eta_{1}^{-1} p_{10}^{(+)}.
\end{eqnarray*}
Hence, 
\begin{eqnarray*}
&&\pi(1,n_{2}) = \frac{p_{10}^{(2)}}{p_{10}^{(+)}}\eta_{1} \eta_{2}^{n_{2}-1}\pi(0,1), \quad n_{2} \ge 1,\\
\end{eqnarray*}
which implies
\begin{eqnarray*}
&&\pi(1,1) = \frac{p_{10}^{(2)}}{p_{10}^{(+)}}\eta_{1} \pi(0,1) = c^{(2+)}\eta_{1} \pi(0,1),\\
&&\pi(n_{1},n_{2}) = \eta_{1}^{n_{1}- 1} \pi(1,n_{2})  = c^{(2+)} \eta_{1}^{n_{1}-1} \eta_{2}^{n_{2}-1} \pi(1,1), \quad (n_{1},n_{2}) \in {\sr S}_{+}.\\
\end{eqnarray*}
On the other hand
\begin{eqnarray*}
&&\tilde{p}_{(-1)(-1)}^{(+)} = \mathbb{P}(\tilde{\vc{Z}}_{\ell + 1}= (0,0) | \tilde{\vc{Z}}_{\ell}= (1,1) ) = \frac{\pi(0,0)}{\pi(1,1)} p_{11}^{(0)} = 0, \\
&&\tilde{p}_{(-1)0}^{(1)} = \mathbb{P}(\tilde{\vc{Z}}_{\ell + 1}= (0,0) | \tilde{\vc{Z}}_{\ell}= (1,0) ) = \frac{\pi(0,0)}{\pi(1,0)} p_{10}^{(0)} = \frac{\pi(n_{1}-1,0)}{\pi(n_{1},0)}p_{10}^{(1)} = \alpha_{1}^{-1} p_{10}^{(1)},\\
&&\tilde{p}_{0(-1)}^{(2)} = \mathbb{P}(\tilde{\vc{Z}}_{\ell + 1}= (0,0) | \tilde{\vc{Z}}_{\ell}= (0,1) ) = \frac{\pi(0,0)}{\pi(0,1)} p_{01}^{(0)} = \frac{\pi(0,n_{2}-1)}{\pi(0,n_{2})}p_{01}^{(2)} = \eta_{2}^{-1} p_{01}^{(2)}.
\end{eqnarray*}
It follows from the irreducible assumption (i), \eqn{singular_condition1}, $p_{i1}^{(1)} = 0$ for any $i=0,\pm1$ and  $p_{1j}^{(2)} = 0$ for $j=\pm1$ that  $p_{10}^{(1)}, p_{01}^{(2)}> 0$. Thus, we immediately have $p_{11}^{(0)} = 0$, $p_{10}^{(0)},p_{01}^{(0)} > 0$ and 
\begin{eqnarray*}
&&\pi(1,0) = \frac{p_{10}^{(0)}}{p_{10}^{(1)}} \alpha_{1} \pi(0,0) = c^{(10)} \alpha_{1} \pi(0,0), \\
&&\pi(0,1) = \frac{p_{01}^{(0)}}{p_{01}^{(2)}} \eta_{2} \pi(0,0) = c^{(20)} \eta_{2} \pi(0,0),\\
&& \pi(1,1) = c^{(2+)}\eta_{1} \pi(0,1) = c^{(20)}c^{(2+)}\eta_{1} \eta_{2} \pi(0,0).
\end{eqnarray*}
We finally obtain $\eta_{1}, \eta_{2}$ and $\alpha_{1}$. We have the following stationary equations.
\begin{eqnarray*}
&&(p_{10}^{(+)} + p_{(-1)0}^{(+)}) \pi(n_{1},n_{2}) = p_{(-1)0}^{(+)} \pi(n_{1}+1,n_{2}) + p_{10}^{(+)} \pi(n_{1}-1,n_{2}), \\
&&(p_{10}^{(1)} + p_{(-1)0}^{(1)})  \pi(n_{1},0) = p_{(-1)0}^{(1)} \pi(n_{1}+1,0) + p_{10}^{(1)} \pi(n_{1}-1,0).
\end{eqnarray*}
From \eqn{singular1_stationary1} and \eqn{singular1_stationary2}, we have 
\begin{eqnarray*}
&&(p_{10}^{(+)} + p_{(-1)0}^{(+)}) \pi(n_{1},n_{2}) = p_{(-1)0}^{(+)}\eta_{1}\pi(n_{1},n_{2}) + p_{10}^{(+)} \eta_{1}^{-1}\pi(n_{1},n_{2}), \\
&&(p_{10}^{(1)} + p_{(-1)0}^{(1)}) \pi(n_{1},0) = p_{(-1)0}^{(1)} \alpha_{1}\pi(n_{1},0) + p_{10}^{(1)} \alpha_{1}\pi(n_{1}-1,0).
\end{eqnarray*}
Thus, we obtain 
\begin{eqnarray*}
\eta_{1} = \frac{p_{10}^{(+)}}{p_{(-1)0}^{(+)}} < 1, \quad \alpha_{1} = \frac{p_{10}^{(1)}}{p_{(-1)0}^{(1)}} < 1.
\end{eqnarray*}
Repeatedly, from the stationary equations, 
\begin{eqnarray*}
&&(p_{01}^{(2)} + p_{0(-1)}^{(2)} + p_{10}^{(2)}) \pi(0,n_{2}) = p_{0(-1)}^{(2)} \pi(0,n_{2} + 1) + p_{01}^{(2)} \pi(0,n_{2} -1) + p_{(-1)0}^{(+)}\pi(1,n_{2}),
\end{eqnarray*}
we have 
\begin{eqnarray*}
(p_{01}^{(2)} + p_{0(-1)}^{(2)} + p_{10}^{(2)}) &=& p_{0(-1)}^{(2)} \eta_{2} + p_{01}^{(2)} \eta_{2}^{-1} + p_{(-1)0}^{(+)} \frac{p_{10}^{(2)}}{p_{10}^{(+)}} \eta_{1}\\
&=& p_{0(-1)}^{(2)} \eta_{2} + p_{01}^{(2)} \eta_{2}^{-1} + p_{(-1)0}^{(+)} \frac{p_{10}^{(2)}}{p_{10}^{(+)}} \frac{p_{10}^{(+)}}{p_{(-1)0}^{(+)}}\\
&=& p_{0(-1)}^{(2)} \eta_{2} + p_{01}^{(2)} \eta_{2}^{-1} + p_{10}^{(2)}.
\end{eqnarray*}
Hence, 
\begin{eqnarray*}
\eta_{2} = \frac{p_{01}^{(2)}}{p_{0(-1)}^{(2)}} < 1.
\end{eqnarray*}

\section{Proof of \lem{con_c1}}
\setnewcounter
\label{app:proof_c1}
We first prove the condition (a1). 
For $n_1 \ge 1$ and $i=0,\pm 1$, substituting $\vc{n} = (n_{1},1)$ and $\vc{n}' = (n_{1} - i, 0)$ into \eqn{RRRW_equation},
\begin{eqnarray*}
\mathbb{P}(\tilde{\vc{Z}}_{\ell+1} = (n_{1}-i,0) | \tilde{\vc{Z}}_{\ell} = (n_{1},1)) = \mathbb{P}(\tilde{\vc{X}}^{(+)} = -(i,1)).
\end{eqnarray*}
Since $\vc{n}' \in \sr{S}_{1}$, we have, using \lem{sta_geo},  
\begin{eqnarray*}
\mathbb{P}(\tilde{\vc{X}}^{(+)} = -(i,1)) &=& \frac{\pi(n_{1} - i,0)}{\pi(n_{1},1)} \mathbb{P}(\vc{Z}_{\ell + 1} = (n_{1},1) | \vc{Z}_{\ell} = (n_{1} - i,0))\\
&=& \frac{\pi(1,0)}{\pi(1,1)}p^{(1)}_{i1}\eta_{1}^{-i}, \quad i=0,\pm1.
\end{eqnarray*}
On the other hand, from \lem{sta_geo}, for $\vc{n},\vc{n}' = \vc{n} - (i,1) \in \sr{S}_{+}$, we have 
\begin{eqnarray*}
\mathbb{P}(\tilde{\vc{X}}^{(+)} = -(i,1)) &=& \frac{\pi(\vc{n} - (i,1))}{\pi(\vc{n})}\mathbb{P}(\vc{X}^{(+)} = (i,1))\\
&=& \eta_{1}^{-i} \eta_{2}^{-1}p_{i1}^{(+)} , \quad i = 0,\pm1.
\end{eqnarray*}
Thus, we obtain, for all $i=0,\pm1$, 
\begin{eqnarray*}
\frac{\pi(1,0)}{\pi(1,1)} p_{i1}^{(1)} = \eta_{2}^{-1}p_{i1}^{(+)}.
\end{eqnarray*}
Thus, $p_{i1}^{(1)} = 0$ if and only if $p_{i1}^{(+)} = 0$.
In addition, by the condition (ii), for some $i = 0,\pm1$ such that $p_{i1}^{(+)} > 0$, 
\begin{eqnarray}
\label{eqn:sta_interior}
p_{i1}^{(1)} = \frac{\pi(1,1)}{\pi(1,0)} \eta_{2}^{-1} p_{i1}^{(+)} = c^{(1+)} p_{i1}^{(+)}.
\end{eqnarray}
Using a similar argument, we also have 
\begin{eqnarray}
\label{eqn:sta_interior2}
p_{1i}^{(2)} = \frac{\pi(1,1)}{\pi(0,1)} \eta_{1}^{-1} {p_{1i}^{(+)}} = c^{(2+)}{p_{1i}^{(+)}}.
\end{eqnarray}
These complete the proof of the condition (a1).

We next obtain the conditions (a2), (a3) and (a4). Since $\{\tilde{\vc{Z}}_{\ell}\}$ has the homogeneous transitions as long as $\tilde{\vc{Z}}_{\ell} \in  \sr{S}_{1}$, we have, from \eqn{transition_of_reversed}, \eqn{RRRW_equation} and \lem{sta_geo}, 
\begin{eqnarray*}
\mathbb{P}(\tilde{\vc{X}}^{(1)} = (-1,0)) = \frac{\pi(0,0)}{\pi(1,0)} p_{10}^{(0)} = \frac{\pi(n-1,0)}{\pi(n,0)} p_{10}^{(1)} = \eta_{1}^{-1} p_{10}^{(1)}, \quad \forall n \ge 1.
\end{eqnarray*}
Similarly, for $\tilde{\vc{X}}^{(+)} = (-1,-1)$ and for any $n \ge 1$, we have 
\begin{eqnarray}
\label{eqn:sta_form1}
\mathbb{P}(\tilde{\vc{X}}^{(+)} = (-1,-1)) = \frac{\pi(0,0)}{\pi(1,1)} p_{11}^{(0)} = \frac{\pi(n-1,0)}{\pi(n,1)} p_{11}^{(1)} = \frac{\pi(1,0)}{\pi(1,1)}\eta_{1}^{-1} p_{11}^{(1)}.
\end{eqnarray}
Hence, for any $i=0,1$, we must have $p_{1i}^{(0)} = p_{1i}^{(1)} = 0$ if $c^{(10)} = 0$, and $c^{(10)} p_{1i}^{(1)}  = p_{1i}^{(0)}$ if $c^{(10)} >0$. 
In addition, if $c^{(10)} > 0$, then we have, from \eqn{sta_interior}--\eqn{sta_form1}
\begin{eqnarray*}
&&\pi(1,0) = c^{(10)} \eta_{1} \pi(0,0), \\
&&\pi(1,1) = c^{(1+)} \eta_{2} \pi(1,0) = c^{(10)} c^{(1+)} \eta_{1} \eta_{2}\pi(0,0),\\
&&\pi(0,1) = \frac{1}{c^{(2+)}} \eta_{1}^{-1} \pi(1,1) = \frac{c^{(10)} c^{(1+)}}{c^{(2+)}}\eta_{2}\pi(0,0).
\end{eqnarray*}
Similarly, we obtain \eqn{stationary_equation_c10=0}--\eqn{stationary_dist11_2} if $c^{(20)} > 0$. We complete the proof.

\section{Proof of \cor{reversibility_product form}}
\label{app:product_form}
\setnewcounter

We first note that the stationary distribution has a product form if and only if
\begin{eqnarray}
\label{eqn:product_form_solution_app}
\pi(n_{1},n_{2}) = \nu_{n_{1}}^{(1)} \nu_{n_{2}}^{(2)}, \quad n_{1},n_{2} \in \mathbb{Z}_{+},
\end{eqnarray}
for some $\nu_{n_1}^{({1})}, \nu_{n_2}^{({2})} \in (0,1)$. Assume that the stationary distribution of $\{\vc{Z}_{\ell}\}$ satisfies \eqn{product_form_solution_app}. Since $\{\vc{Z}_{\ell}\}$ is structure-reversible, if $c^{(20)} > 0$, then from \eqn{stationary_equation_c10=0}--\eqn{stationary_dist11_2}, 
\begin{eqnarray}
\label{eqn:product}
&& \nu_{0}^{(1)} \nu_{0}^{(2)} = \pi(0,0),\\
\label{eqn:product_1}
&& \nu_{1}^{(1)} \nu_{0}^{(2)} = \frac{c^{(20)}c^{(2+)}}{c^{(1+)}} \eta_{1} \pi(0,0), \\
\label{eqn:product_2}
&& \nu_{0}^{(1)} \nu_{1}^{(2)} =  c^{(20)}\eta_{2} \pi(0,0),\\
\label{eqn:product_3}
&&\nu_{1}^{(1)} \nu_{1}^{(2)} = c^{(20)}c^{(2+)} \eta_{1}\eta_{2} \pi(0,0).
\end{eqnarray}
Substituting \eqn{product} and \eqn{product_1} into \eqn{product_2} and \eqn{product_3}, we have 
\begin{eqnarray*} 
\nu_{1}^{(2)} = c^{(20)} \eta_{2} \nu_{0}^{(2)} =  c^{(1+)}\eta_{2} \nu_{0}^{(2)}.
\end{eqnarray*}
Hence, the condition \eqn{Condition_product1} holds. We also have \eqn{Condition_product2} in a similar way.

We next suppose that the condition \eqn{Condition_product1} is satisfied. Then, we have 
\begin{align*}
&\pi(1,0) =\frac{c^{(20)}c^{(2+)}}{c^{(1+)}} \eta_{1} \pi(0,0) = c^{(2+)} \eta_{1} \pi(0,0).
\end{align*}
Thus, we can rewrite the stationary distribution of \thr{reversibility condition} as 
\begin{align*}
&\pi(n,0) = \eta_{1}^{n-1} \pi(1,0) = c^{(2+)} \eta_{1}^{n}\pi(0,0), &&n \ge 1,\\
&\pi(0,n) = \eta_{2}^{n-1} \pi(0,1) = c^{(20)} \eta_{2}^{n}\pi(0,0), &&n \ge 1,\\
&\pi(n_{1},n_{2}) = \eta_{1}^{n_{1}-1}\eta_{2}^{n_{2}-1} \pi(1,1) = c^{(20)}c^{(2+)} \eta_{1}^{n_{1}} \eta_{2}^{n_{2}} \pi(0,0), && n_{1},n_{2} \ge 1.
\end{align*}
For $n \in \mathbb{Z}_{+}$, we put $\nu_{n}^{(1)}$ and $\nu_{n}^{(2)}$ as follows.
\begin{align*}
&\nu_{n}^{(1)} = c^{(2+)} \eta_{1}^{n} \nu_{0}^{(1)}, && n \ge 1,\\
&\nu_{n}^{(2)} = c^{(20)} \eta_{2}^{n} \nu_{0}^{(2)}, && n \ge 1,\\
&\nu_{0}^{(1)}\nu_{0}^{(2)} = \pi(0,0), 
\end{align*}
and we have
\begin{align*}
& \nu_{n}^{(1)} \nu_{0}^{(2)} = c^{(2+)} \eta_{1}^{n} \pi(0.0) = \pi(n,0), &&n \ge 1,\\
& \nu_{0}^{(1)} \nu_{n}^{(2)} = c^{(20)} \eta_{2}^{n} \pi(0,0) = \pi(0,n), &&n \ge 1,\\
& \nu_{n_{1}}^{(1)} \nu_{n_{2}}^{(2)} = c^{(20)}c^{(2+)} \eta_{1}^{n_{1}}\eta_{2}^{n_{2}}\pi(0,0) = \pi(n_{1},n_{2}), &&n_{1},n_{2} \ge 1.\\
\end{align*} 
These equations imply \eqn{product_form_solution_app}. We similarly obtain \eqn{product_form_solution_app} if the condition \eqn{Condition_product2} holds.
 This completes the proof since we have either $c^{(10)} > 0$ or $c^{(20)} > 0$ (see \rem{irreducibility}).

\section{Product form but not structure-reversibility}
\label{app:Example of product}
\setnewcounter
We give an example of the reflecting random walk to have a product form stationary distribution but not to be structure-reversible. For transition probabilities of the reflecting random walk, let
\begin{align*}
&p_{00}^{(0)} = 0.0840821, && p_{10}^{(0)} = 0.49716, && p_{01}^{(0)} = 0.188503, && p_{11}^{(0)} = 0.230255,\\
&p_{00}^{(1)} = 0.0840821, && p_{10}^{(1)} = 0.126693, && p_{01}^{(1)} = 0.216346, && p_{11}^{(1)} = 0.15534, \\
&p_{(-1)0}^{(1)} = 0.1205,&& p_{(-1)1}^{(1)} = 0.297039,&& p_{00}^{(2)} = 0.363151, && p_{10}^{(2)} = 0.267565,\\
&p_{01}^{(2)} = 0.0246397,&& p_{11}^{(2)} = 0.025552, && p_{0(-1)}^{(2)} = 0.223309, && p_{1(-1)}^{(2)} = 0.0957827,\\
&p_{00}^{(+)} = 0.469511,&& p_{10}^{(+)} = 0.0449179, && p_{01}^{(+)} = 0.00497654, && p_{11}^{(+)} = 0.012212,\\
&p_{(-1)0}^{(+)} = 0.398019,&& p_{(-1)1}^{(+)} = 0.0045338,&& p_{0(-1)}^{(+)} = 0.0380278,&& p_{1(-1)}^{(+)} = 0.0278023,\\
&p_{(-1)(-1)}^{(+)} = 0.&&&&&&
\end{align*}
Then, we can see that (3.13), (3.14), (3.17) and (3.24)--(3.29) of \cite{LatoMiya2013} hold with $\eta_{1} = 0.3$ and $\eta_{2} = 0.2$, and therefore, the stationary distribution of this random walk has a product form. For this random walk, we have 
\begin{eqnarray*}
\frac{p_{(-1)1}^{(1)}}{p_{(-1)1}^{(+)}} = 65,5165, \qquad \frac{p_{01}^{(1)}}{p_{01}^{(+)}} = 43.4732, \qquad \frac{p_{11}^{(1)}}{p_{11}^{(+)}} = 12.7203.
\end{eqnarray*}
Hence, (a1) does not hold, and therefore, this random walk is not structure-reversible.

\section{Proof of \eqn{check_stationary1}}
\label{app:check_stationary1}
\setnewcounter
Recall that $c^{(10)} > 0$. Similar to the case $\vc{n} = (0,0)$, using the conditions (a1)--(a3), \eqn{stationary_dist2_1}--\eqn{stationary_dist11}, and \eqn{remark_c01}, we compute the left hand side of \eqn{check_stationary1} in the following ways.
\begin{itemize}
\item For $\vc{n} = ( 1 , 0 )$, we have $\vc{n}' \in \{(i,j); i=0,1,2,j=0,1\}$.
\begin{eqnarray*}
\label{eqn:check_10}
\lefteqn{\sum_{\vc{n'} \in {\sr S}}\frac{\pi({\vc{n'}})}{\pi(\vc{n})} \mathbb{P}(\vc{Z}_{\ell+1} = \vc{n} | \vc{Z}_{\ell} = \vc{n}') = p_{00}^{(1)} + \frac{\pi(0 , 0)}{\pi(1 , 0)} p_{10}^{(0)}  + \frac{\pi(2 , 0)}{\pi(1 , 0)} p_{(-1)0}^{(1)}}\nonumber\\
&& \qquad \qquad + \frac{\pi(0 , 1)}{\pi(1 , 0)} p_{1(-1)}^{(2)} + \frac{\pi(1 , 1)}{\pi(1 , 0)} p_{0(-1)}^{(+)} + \frac{\pi(2 , 1)}{\pi(1 , 0)} p_{(-1)(-1)}^{(+)} \nonumber\\
&&\quad = p_{00}^{(1)} + \frac{1}{c^{(10)}} p_{10}^{(0)} \eta_{1}^{-1}  + p_{(-1)0}^{(1)} \eta_{1} \nonumber\\
&& \qquad \qquad+ \frac{c^{(1+)}}{c^{(2+)}} p_{1(-1)}^{(2)} \eta_{1}^{-1} \eta_{2} + c^{(1+)} p_{0(-1)}^{(+)} \eta_{2} + c^{(1+)} p_{(-1)(-1)}^{(+)} \eta_{1} \eta_{2}\nonumber\\
&&\quad = p_{00}^{(1)} + p_{10}^{(1)} \eta_{1}^{-1} + p_{(-1)0}^{(1)} \eta_{1} + c^{(1+)} \left( p_{1(-1)}^{(+)} \eta_{1}^{-1} \eta_{2} + p_{0(-1)}^{(+)} \eta_{2} + p_{(-1)(-1)}^{(+)} \eta_{1} \eta_{2} \right)\nonumber\\
&&\quad = \gamma_{1} (\eta_{1}^{-1} , \eta_{2}^{-1} )=1.
\end{eqnarray*}

\item For $\vc{n} = ( 1 , 1 )$, we have $\vc{n}' \in \{(i,j); i,j = 0,1,2\}$.
\begin{eqnarray*}
&&\sum_{\vc{n'} \in {\sr S}}\frac{\pi({\vc{n'}})}{\pi(\vc{n})} \mathbb{P}(\vc{Z}_{\ell+1} = \vc{n} | \vc{Z}_{\ell} = \vc{n}') \\
&& \quad = p_{00}^{(+)} + \frac{\pi( 0 , 1 )}{\pi( 1 , 1 )} p_{10}^{(2)} + \frac{\pi( 0 , 2 )}{\pi( 1 , 1 )} p_{1(-1)}^{(2)} + \frac{\pi( 1 , 2 )}{\pi( 1 , 1 )} p_{0(-1)}^{(+)} + \frac{\pi( 2 , 2 )}{\pi( 1 , 1 )} p_{(-1)(-1)}^{(+)} \\
&&\quad \quad \quad + \frac{\pi( 2 , 1 )}{\pi( 1 , 1 )} p_{(-1)0}^{(+)} + \frac{\pi( 2 , 0 )}{\pi( 1 , 1 )} p_{(-1)1}^{(1)} + \frac{\pi( 1 , 0 )}{\pi( 1 , 1 )} p_{01}^{(1)} + \frac{\pi( 0 , 0 )}{\pi( 1 , 1 )} p_{11}^{(0)}\\
&&\quad = p_{00}^{(+)} + \frac{1}{c^{(2+)}} p_{10}^{(2)} \eta_{1}^{-1} +\frac{1}{c^{(2+)}} p_{1(-1)}^{(2)} \eta_{1}^{-1} \eta_{2} + p_{0(-1)}^{(+)} \eta_{2} + p_{(-1)(-1)}^{(+)} \eta_{1} \eta_{2}  \\
&&\quad \quad \quad + p_{(-1)0}^{(+)} \eta_{1} + \frac{1}{c^{(1+)}}  p_{(-1)1}^{(1)} \eta_{1} \eta_{2}^{-1} + \frac{1}{c^{(1+)}} p_{01}^{(1)} \eta_{2}^{-1} + \frac{1}{c^{(10)}c^{(1+)}} p_{11}^{(0)} \eta_{1}^{-1} \eta_{2}^{-1}\\
&&\quad = p_{00}^{(+)} + p_{10}^{(+)} \eta_{1}^{-1} + p_{1(-1)}^{(+)}  \eta_{1}^{-1} \eta_{2} + p_{0(-1)}^{(+)} \eta_{2} + p_{(-1)(-1)}^{(+)} \eta_{1} \eta_{2} + p_{(-1)0}^{(+)} \eta_{1} \\
&&\quad \quad \quad +  p_{(-1)1}^{(+)} \eta_{1} \eta_{2}^{-1} + p_{01}^{(+)} \eta_{2}^{-1} + p_{11}^{(+)} \eta_{1}^{-1} \eta_{2}^{-1}\\
&&\quad = \gamma_{+} (\eta_{1}^{-1} , \eta_{2}^{-1}) = 1.
\end{eqnarray*}

\item For $\vc{n} = ( n_{1} , 0 )$, we have $\vc{n}' \in \{(n_{1}+i,j), i=0,\pm1,j=0,1\}$.
\begin{eqnarray*}
&&\sum_{\vc{n'} \in {\sr S}}\frac{\pi({\vc{n'}})}{\pi(\vc{n})} \mathbb{P}(\vc{Z}_{\ell+1} = \vc{n} | \vc{Z}_{\ell} = \vc{n}') \\
&& \quad = p_{00}^{(1)} + \frac{\pi( n_{1} - 1 , 0 )}{\pi( n_{1} , 0 )} p_{10}^{(1)} + \frac{\pi( n_{1} + 1 , 0 )}{\pi( n_{1} , 0 )} p_{(-1)0}^{(1)} \\
&& \qquad \qquad + \frac{\pi( n_{1} - 1 , 1 )}{\pi( n_{1} , 0 )} p_{1(-1)}^{(+)} + \frac{\pi( n_{1}  , 1 )}{\pi( n_{1} , 0 )} p_{0(-1)}^{(+)} + \frac{\pi( n_{1} + 1 , 1 )}{\pi( n_{1} , 0 )} p_{(-1)(-1)}^{(+)}\\
&&\quad = p_{00}^{(1)} + p_{10}^{(1)} \eta_{1}^{-1} + p_{(-1)0}^{(1)} \eta_{1} + c^{(1+)} \left( p_{1(-1)}^{(+)} \eta_{1}^{-1} \eta_{2} + p_{0(-1)}^{(+)} \eta_{2} + p_{(-1)(-1)}^{(+)} \eta_{1} \eta_{2} \right)\\
&&\quad = \gamma_{1} (\eta_{1}^{-1} , \eta_{2}^{-1}) = 1.
\end{eqnarray*}

\item For $\vc{n} = ( n_{1} , 1 )$, we have $\vc{n}' \in \{(n_{1} + i,j)\; i=0,\pm1,j=0,1,2\}$.
\begin{eqnarray*}
&&\sum_{\vc{n'} \in {\sr S}}\frac{\pi({\vc{n'}})}{\pi(\vc{n})} \mathbb{P}(\vc{Z}_{\ell+1} = \vc{n} | \vc{Z}_{\ell} = \vc{n}') \\
&& \quad = p_{00}^{(+)} + \frac{\pi( n_{1} - 1 , 1 )}{\pi( n_{1} , 1 )} p_{10}^{(+)} + \frac{\pi( n_{1} - 1 , 2 )}{\pi( n_{1} , 1 )} p_{1(-1)}^{(+)} \\
&&\qquad \qquad   + \frac{\pi( n_{1} , 2 )}{\pi( n_{1} , 1 )} p_{0(-1)}^{(+)} + \frac{\pi( n_{1} + 1 , 2 )}{\pi( n_{1} , 1 )} p_{(-1)(-1)}^{(+)} + \frac{\pi( n_{1} + 1 , 1 )}{\pi( n_{1} , 1 )} p_{(-1)0}^{(+)}\\
&& \qquad \qquad   + \frac{\pi( n_{1} + 1 , 0 )}{\pi( n_{1} , 1 )} p_{(-1)1}^{(1)} + \frac{\pi( n_{1} , 0 )}{\pi( n_{1} , 1 )} p_{01}^{(1)} + \frac{\pi( n_{1} - 1 , 0 )}{\pi( n_{1} , 1 )} p_{11}^{(1)}\\
&&\quad = p_{00}^{(+)} + p_{10}^{(+)} \eta_{1}^{-1} + p_{1(-1)}^{(+)} \eta_{1}^{-1} \eta_{2} + p_{0(-1)}^{(+)} \eta_{2} + p_{(-1)(-1)}^{(+)} \eta_{1} \eta_{2}  \\
&&\qquad \qquad + p_{(-1)0}^{(+)} \eta_{1} + \frac{1}{c^{(1+)}} p_{(-1)1}^{(1)} \eta_{1} \eta_{2}^{-1} + \frac{1}{c^{(1+)}} p_{01}^{(1)} \eta_{2}^{-1} + \frac{1}{c^{(1+)}} p_{11}^{(1)} \eta_{1}^{-1} \eta_{2}^{-1}\\
&&\quad = p_{00}^{(+)} + p_{10}^{(+)} \eta_{1}^{-1} + p_{1(-1)}^{(+)} \eta_{1}^{-1} \eta_{2} + p_{0(-1)}^{(+)} \eta_{2} + p_{(-1)(-1)}^{(+)} \eta_{1} \eta_{2}  \\
&&\qquad \qquad + p_{(-1)0}^{(+)} \eta_{1}+ p_{(-1)1}^{(+)}  \eta_{1} \eta_{2}^{-1} +  p_{01}^{(+)} \eta_{2}^{-1} +  p_{11}^{(+)} \eta_{1}^{-1} \eta_{2}^{-1} \\
&&\quad = \gamma_{+} (\eta_{1}^{-1} , \eta_{2}^{-1})=1.
\end{eqnarray*}

\item For $\vc{n} = ( n_{1} , n_{2} )$, we have $\vc{n}' \in \{(n_{1}+i,n_{2}+j);i,j=0,\pm1\}$.
\begin{eqnarray*}
&&\sum_{\vc{n'} \in {\sr S}}\frac{\pi({\vc{n'}})}{\pi(\vc{n})} \mathbb{P}(\vc{Z}_{\ell+1} = \vc{n} | \vc{Z}_{\ell} = \vc{n}') \\
&& \quad = p_{00}^{(+)} + \frac{ \pi(n_{1} - 1 , n_{2} ) }{ \pi(n_{1} , n_{2}) } p_{10}^{(+)} + \frac{ \pi(n_{1} - 1 , n_{2} + 1 ) }{ \pi(n_{1} , n_{2}) } p_{1(-1)}^{(+)} \\
&&\qquad \qquad + \frac{ \pi(n_{1} , n_{2} + 1 ) }{ \pi(n_{1} , n_{2}) } p_{0(-1)}^{(+)} + \frac{ \pi(n_{1} + 1 , n_{2} + 1 ) }{ \pi(n_{1} , n_{2}) } p_{(-1)(-1)}^{(+)} + \frac{ \pi(n_{1} + 1 , n_{2} ) }{ \pi(n_{1} , n_{2}) } p_{(-1)0}^{(+)} \\
&& \qquad \qquad + \frac{ \pi(n_{1} + 1 , n_{2} - 1 ) }{ \pi(n_{1} , n_{2}) } p_{(-1)1}^{(+)} + \frac{ \pi(n_{1} , n_{2} - 1 ) }{ \pi(n_{1} , n_{2}) } p_{01}^{(+)} + \frac{ \pi(n_{1} - 1 , n_{2} - 1 ) }{ \pi(n_{1} , n_{2}) } p_{11}^{(+)}\\
&&\quad = p_{00}^{(+)} + p_{10}^{(+)} \eta_{1}^{-1} + p_{1(-1)}^{(+)} \eta_{1}^{-1} \eta_{2} + p_{0(-1)}^{(+)} \eta_{2} + p_{(-1)(-1)}^{(+)} \eta_{1} \eta_{2} \\
&&\qquad \qquad + p_{(-1)0}^{(+)} \eta_{1} + p_{(-1)1}^{(+)} \eta_{1} \eta_{2}^{-1} + p_{01}^{(+)} \eta_{2}^{-1} + p_{11}^{(+)} \eta_{1}^{-1} \eta_{2}^{-1}\\
&&\quad = \gamma_{+} (\eta_{1}^{-1} , \eta_{2}^{-1}) = 1.
\end{eqnarray*}

\end{itemize}

\section{Proof of $\gamma_{i}(\rho_{1}^{-1},\rho_{2}^{-1}) = 1$}
\label{app:computation}
\setnewcounter
Under the assumption \eqn{parameter_condition_1}, we rewrite the traffic equation \eqn{traffic_equation} as follows.
\begin{eqnarray}
\label{eqn:traffic_equation_2}
\alpha_{1} = \lambda  + \alpha_{2} r_{21} + \alpha_{1} r_{11}, \quad \alpha_{2} = \lambda  + \alpha_{1} r_{12} + \alpha_{2} r_{22}.
\end{eqnarray}
The solutions of these equations are given by
\begin{eqnarray*}
\alpha_{1} = \frac{\lambda(1 - r_{22}) + \lambda r_{21}}{1 - r_{11} - r_{22} - r_{12} r_{21} + r_{11}r_{22}}, \quad \alpha_{2} = \frac{\lambda(1 - r_{11}) + \lambda r_{12}}{1 - r_{11} - r_{22} - r_{12} r_{21} + r_{11}r_{22}}.
\end{eqnarray*}
We note that $1 - r_{11}  = r_{12} + r_{10}$ and $1 - r_{22}  = r_{21} + r_{20}$. From the assumption \eqn{parameter_condition_2}, 
\begin{eqnarray*}
(1 - r_{22})r_{12} = (1 - r_{11}) r_{21}, 
\end{eqnarray*}
and therefore, we have $\alpha_{1}r_{12} = \alpha_{2} r_{21}$. Substituting this equation into \eqn{traffic_equation_2}, we have $\lambda = \alpha_{1} r_{10} = \alpha_{2} r_{20}$. For $i=0,1,+$, 
\begin{eqnarray*}
\gamma_{0}(\rho_{1}^{-1},\rho_{2}^{-1})
&=&\mu_{1}+\mu_{2}-\lambda_{1}^{(0)}-\lambda_{2}^{(0)}+\Bigl(1+\frac{\lambda_{1}^{(0)}}{\lambda}\Bigr)\mu_{1}r_{10}\rho_{1}+\Bigl(1+\frac{\lambda_{2}^{(0)}}{\lambda}\Bigr)\mu_{2}r_{20}\rho_{2}\\
&=&\mu_{1}+\mu_{2}-\lambda_{1}^{(0)}-\lambda_{2}^{(0)}+\Bigl(1+\frac{\lambda_{1}^{(0)}}{\lambda}\Bigr)\alpha_{1}r_{10}+\Bigl(1+\frac{\lambda_{2}^{(0)}}{\lambda}\Bigr)\alpha_{2}r_{20}\\
&=&\mu_{1}+\mu_{2} + \lambda + \lambda = 1,\\
\gamma_{1} (\rho_{1}^{-1} , \rho_{2}^{-1}) 
 &=& \mu_{2} - \lambda_{2}^{(1)} + \rho_{1}^{-1} \lambda + \mu_{1} r_{10} \rho_{1} + \left( 1 + \frac{\lambda_{2}^{(1)}}{\lambda} \right) \left( \mu_{2} r_{21} \rho_{1}^{-1} \rho_{2} + \mu_{2} r_{20} \rho_{2} \right)\\
 &=& \mu_{2} - \frac{r_{11}}{r_{12}} \lambda + \frac{\mu_{1}}{\alpha_{1}} \lambda + \lambda + \left( 1 + \frac{r_{11}}{r_{12}} \right) \left( \frac{\alpha_{2}}{\alpha_{1}} r_{21} \mu_{1} + r_{20} \alpha_{2} \right)\\
&=& \mu_{2} - \frac{r_{11}}{r_{12}} \lambda + \mu_{1} r_{10} + \lambda + \mu_{1} r_{12} + \lambda + \mu_{1} r_{11} + \frac{r_{11}}{r_{12}} \lambda= 1,\\\gamma_{+} ( \rho_{1}^{-1} , \rho_{2}^{-1} ) 
 &=& \mu_{1} r_{11} + \mu_{2} r_{22} + \mu_{1} r_{10} + \mu_{2} r_{20} +\lambda + \mu_{2} r_{21} + \lambda + \mu_{1} r_{12} = 1.
\end{eqnarray*}
By symmetry of $\gamma_{1}(\rho_{1}^{-1},\rho_{2}^{-1}) = 1$, we have $\gamma_{2}(\rho_{1}^{-1},\rho_{2}^{-1}) = 1$. 

\bibliographystyle{ims}
\bibliography{koba04042013}

\def\cprime{$'$}
\begin{thebibliography}{10}
\expandafter\ifx\csname natexlab\endcsname\relax\def\natexlab#1{#1}\fi
\expandafter\ifx\csname url\endcsname\relax
  \def\url#1{\texttt{#1}}\fi
\expandafter\ifx\csname urlprefix\endcsname\relax\def\urlprefix{URL }\fi
\providecommand{\eprint}[2][]{\url{#2}}

\bibitem[{Asmussen(2003)}]{Asmu2003}
\textsc{Asmussen, S.} (2003).
\newblock \textit{Applied probability and queues}, vol.~51 of
  \textit{Applications of Mathematics}.
\newblock 2nd ed. Springer-Verlag, New York.
\newblock Stochastic Modelling and Applied Probability.

\bibitem[{Chao et~al.(1999)Chao, Miyazawa and Pinedo}]{ChaoMiyaPine1999}
\textsc{Chao, X.}, \textsc{Miyazawa, M.} and \textsc{Pinedo, M.} (1999).
\newblock \textit{Queueing networks, customers, signals and product form
  solutions}.
\newblock John Wiley \& Sons Inc., New York.

\bibitem[{Fayolle et~al.(1999)Fayolle, Iasnogorodski and
  Malyshev}]{FayoIasnMaly1999}
\textsc{Fayolle, G.}, \textsc{Iasnogorodski, R.} and \textsc{Malyshev, V.}
  (1999).
\newblock \textit{Random Walks in the Quarter-Plane: Algebraic Methods,
  Boundary Value Problems and Applications}.
\newblock Springer, New York.

\bibitem[{Kelly(1979)}]{Kell1979}
\textsc{Kelly, F.~P.} (1979).
\newblock \textit{Reversibility and Stochastic Networks}.
\newblock New York, John Wiley \& Sons Inc.

\bibitem[{Kobayashi and Miyazawa(2012)}]{KobaMiya2012}
\textsc{Kobayashi, M.} and \textsc{Miyazawa, M.} (2012).
\newblock Revisit to the tail asymptotics of the double {QBD} process:
  Refinement and complete solutions for the coordinate and diagonal directions.
\newblock In \textit{Matrix-Analytic Methods in Stochastic Models} (G.~Latouche
  and M.~S. Squillante, eds.). Springer, 147--181.
\newblock ArXiv:1201.3167.

\bibitem[{Latouche and Miyazawa(2013)}]{LatoMiya2013}
\textsc{Latouche, G.} and \textsc{Miyazawa, M.} (2013).
\newblock Product form characterization for a two dimensional reflecting random
  walk and its applications.
\newblock To appear Queueing systems,
  \urlprefix\url{http://link.springer.com/article/10.1007/s11134-013-9381-7}.

\bibitem[{Miyazawa(1997)}]{Miya1997}
\textsc{Miyazawa, M.} (1997).
\newblock Structure-reversibility and departure functions of queueing networks
  with batch movements and state dependent routing.
\newblock \textit{Queueing Systems}, \textbf{25} 45--75.

\bibitem[{Miyazawa(2011)}]{Miya2011}
\textsc{Miyazawa, M.} (2011).
\newblock Light tail asymptotics in multidimensional reflecting processes for
  queueing networks.
\newblock \textit{TOP}, \textbf{19} 233--299.

\bibitem[{Miyazawa(2013)}]{Miya2012}
\textsc{Miyazawa, M.} (2013).
\newblock \textit{Reversibility in Queueing Models}.
\newblock Wiley Encyclopedia of Operations Research and Management Science, New
  York.

\bibitem[{Serfozo(1999)}]{Serf1999}
\textsc{Serfozo, R.} (1999).
\newblock \textit{Introduction to stochastic networks}, vol.~44 of
  \textit{Applications of Mathematics}.
\newblock Springer-Verlag, New York.

\end{thebibliography}
\end{document}